\documentclass[pdf]{article}

\usepackage{pdfpages}
\usepackage{hyperref}
\usepackage{xcolor}
\usepackage[utf8]{inputenc}
\usepackage{amsmath}
\usepackage{amsfonts}
\usepackage{amssymb}
\usepackage{amsthm}
\usepackage{array}
\usepackage{marginnote}
\usepackage{enumitem}
\usepackage{graphicx}
\usepackage[normalem]{ulem}
\usepackage{fancyhdr}

	\usepackage{pst-node}
	\usepackage{tikz-cd}
	\usetikzlibrary{calc}
\usepackage{ulem}
\usepackage{cleveref}

\newcommand{\m}{\mathrm{M}}
\newcommand{\mg}{\mathrm{MG}}

\newcommand{\defstyle}[1]{\textbf{#1}}
\newcommand{\myprob}[1]{\mathbb P \left[ #1 \right]}
\newcommand{\probPalm}[2]{\mathbb P_{#1} \left[ #2 \right]}

\newcommand{\omid}[1]{\mathbb E \left[ #1 \right]}
\newcommand{\omidPalm}[2]{\mathbb E_{#1} \left[ #2 \right]}
\newcommand{\omidCond}[2]{\mathbb E \left[ #1 \left| #2 \right. \right]}

\newcommand{\norm}[1]{\left| #1 \right|}

\newcommand{\identity}[1]{1_{#1}}
\newcommand{\bs}[1]{\boldsymbol{#1}}
\newcommand{\card}[1]{\left|#1\right|}
\newcommand{\cost}{\mathrm{Cost}}

\newcommand{\floor}[1]{\left\lfloor #1 \right\rfloor}
\newcommand{\ceil}[1]{\left\lceil #1 \right\rceil}

\newcommand{\del}[1]{}
\newcommand{\unwritten}[1]{}

\newcommand{\oball}[2]{B_{#1}(#2)}

\newcommand{\restrict}[2]{{
		\left.\kern-\nulldelimiterspace 
		#1 
		\vphantom{\big|} 
		\right|_{#2} 
}}

\theoremstyle{theorem}
\newtheorem{theorem}{Theorem}[section]
\newtheorem{lemma}[theorem]{Lemma}
\newtheorem{proposition}[theorem]{Proposition}

\newtheorem{problem}[theorem]{Problem}

\theoremstyle{definition}
\newtheorem{definition}[theorem]{Definition}

\newtheorem{example}[theorem]{Example}
\theoremstyle{definition}
\newtheorem{remark}[theorem]{Remark}

\crefname{model}{Model}{Models}
\crefname{submodel}{Model}{Models}
\crefname{algorithm}{Algorithm}{Algorithms}
\crefname{equation}{}{}
\crefname{problem}{Problem}{Problems}
\crefname{theorem}{Theorem}{Theorems}
\crefname{corollary}{Corollary}{Corollaries}
\crefname{proposition}{Proposition}{Propositions}

\theoremstyle{theorem}

\numberwithin{equation}{section}

\makeatletter
\let\orgdescriptionlabel\descriptionlabel
\renewcommand*{\descriptionlabel}[1]{%
	\let\orglabel\label
	\let\label\@gobble
	\phantomsection
	\edef\@currentlabel{#1}%
	\let\label\orglabel
	\orgdescriptionlabel{#1}%
}

\newlist{stepC}{enumerate}{1}
\setlist[stepC,1]{
	label=\textbf{Step C\arabic*},
	ref=C\arabic*,
	leftmargin=* 
}
\crefname{stepCi}{Step}{Steps}
\Crefname{stepCi}{Step}{Steps}

\newlist{stepCC}{enumerate}{1}
\setlist[stepCC,1]{
	label=\textbf{Step C\arabic{stepCi}.\arabic*},
	ref=C\arabic{stepCi}.\arabic*,
	leftmargin=* 
}
\crefname{stepCCi}{Step}{Steps}
\Crefname{stepCCi}{Step}{Steps}

\newlist{stepCCprime}{enumerate}{1}
\setlist[stepCCprime,1]{
	label=\textbf{Step C\arabic{stepCi}.\arabic*$'$},
	ref=C\arabic{stepCi}.\arabic*$'$,
	leftmargin=* 
}
\crefname{stepCCprimei}{Step}{Steps}
\Crefname{stepCCprimei}{Step}{Steps}

\newlist{stepL}{enumerate}{1}
\setlist[stepL,1]{
	label=\textbf{Step L\arabic*},
	ref=L\arabic*,
	leftmargin=* 
}
\crefname{stepLi}{Step}{Steps}
\Crefname{stepLi}{Step}{Steps}

\newlist{stepLL}{enumerate}{1}
\setlist[stepLL,1]{
	label=\textbf{Step L\arabic{stepLi}.\arabic*},
	ref=L\arabic{stepLi}.\arabic*,
	leftmargin=* 
}
\crefname{stepLLi}{Step}{Steps}
\Crefname{stepLLi}{Step}{Steps}

\makeatother

\begin{document}

\title{Products of Infinite Countable Groups Have Fixed Price One}

\author{Ali Khezeli \footnote{School of Mathematics, Institute for Research in Fundamental Sciences, Tehran, Iran, alikhezeli@ipm.ir}}

\maketitle

\begin{abstract}
	We prove that the product of any two infinite countable groups has fixed price one. This resolves a longstanding problem posed by Gaboriau. The proof uses the propagation method to construct a Poisson horoball process as a weak limit of a sequence of factors of iid. We then construct a low-cost graphing by showing that the resulting horoballs have a variant of the infinite touching property almost surely, if the metric and the other parameters of the construction are chosen carefully. A novelty is providing direct simple proofs that do not rely on sophisticated results like amenability and double-recurrence, which are used in related works. An essential tool for avoiding any growth conditions is the convergence in the sense of point processes of pointed closed subsets, which is a notion from stochastic geometry. Also, to manage the overlapping of the horoballs, a generalization of the induction lemma is presented for random multisets of a group.
%
\end{abstract}


\section{Introduction}

\subsection{Cost and the Fixed Price Property}

Cost is a central notion in measured group theory, introduced in~\cite{Le95cost}. It was substantially developed by Gaboriau in the seminal paper~\cite{Ga00cost} and is used therein to solve several problems in orbit equivalence theory. Roughly speaking, the \defstyle{cost} of a countable group $G$ is the infimum \textit{number of edges per vertex} required for connecting all of the points of $G$ in a stationary random manner. The cost of an action of $G$ is also defined similarly, see \Cref{intro:weakFactor}, and is a measure theoretic analogue of the minimal number of generators.
Cost has proved to be an important invariant of group actions by being connected to several other notions, like orbit equivalence theory, geometric group theory, rank gradient, percolation on graphs, uniform spanning forests, point process theory, $l^2$-Betti numbers, and operator algebras. There is active research on this topic and its open problems, some of which are mentioned below. See also Gaboriau's under progress notes~\cite{Ga24}.
This notion is also extended to locally compact unimodular groups in~\cite{Ca23,AbMe22}, where the second reference uses point processes theory.

The group $G$ has \defstyle{fixed price} if all essentially free actions of $G$ have the same cost. 
It is a fundamental open problem, dating back to the origins of the theory~\cite{Ga00cost}, to determine whether all countable groups have fixed price.
We mention some of the previous works in this matter.
In~\cite{Ga00cost}, the cost of several classes of groups are computed; e.g., finite groups (fixed price $1-1/\norm{G}$), free group with $n$ generators (fixed price $n$), $SL(n,\mathbb Z)$ and fundamental groups of compact surfaces. It is also proved in~\cite{Ga00cost} that lattices in higher rank semi-simple real Lie groups have cost 1. It remained open whether the latter have fixed price 1 until very recently, proved in the landmark~\cite{FMW} using novel methods explained below.
Some identities are also proved in~\cite{Ga00cost}; e.g., for finite-index subgroups (see the \textit{induction lemma}) and free products. It is also proved that a \textit{treeable} group (i.e., a group that admits a graphing which is almost surely a tree) has fixed price. Infinite groups have cost at least 1, and this lower bound is improved in~\cite{Ga02invariants} to 1 plus the first $l^2$-Betti number (which is equal to the cost of the \textit{free uniform spanning forest}; see~\cite{bookLyPe16}). It is still open whether the last lower bound is always equal to the cost. 
Several classes of groups have been proved to have fixed price 1; e.g., amenable groups and groups with an infinite normal subgroup (e.g., with infinite center). Recently, \cite{HuPe20Kazhdan} proved that Kazhdan groups have cost 1. The latter is extended to unimodular random graphs with property (T) in~\cite{GrJaMe25cost}.

It is also proved in~\cite{Ga00cost} that the direct product of any pair of infinite countable groups has cost 1 (see \Cref{intro:motivation} below). 
The fixed price problem for such products was still open until now, despite being an active problem in the field.\footnote{See e.g., the list of problems of the 2011 AIM workshop \textit{$l^2$ invariants and their relatives for finitely generated groups}, available at \hyperlink{http://aimpl.org/l2invariantsgroups/2}{http://aimpl.org/l2invariantsgroups/2}.} It has been proved in the special cases where one of the groups is amenable, or contains an element with infinite order, or contains arbitrarily large finite subgroups, or contains an infinite subgroup which has fixed price 1 (see \cite{Ga00cost,Ga24}), but the general case is much more difficult. The recent papers~\cite{AbMe22,FMW,Me23} use novel methods and prove the claim for further classes of products of groups. We will briefly explain these works in this section and study their connections with the present work. 

In this paper, we fully resolve the fixed price problem for direct products of countable infinite groups:

\begin{theorem}
	\label{thm:fixedpriceGeneral}
	The product of any two infinite countable groups has fixed price one.
\end{theorem}

The proof is provided in \Cref{sec:proof}. We will first prove this claim for finitely generated groups (\Cref{thm:fixedpriceFG}). Gaboriau pointed out\footnote{Personal communication in September 2025 after the first version of the paper.} that this can be easily extended to products of general countable groups using already existing results. Before giving the proofs, \Cref{intro:sketch} describes the intuitions of the main steps and ideas, together with providing further results.

In the ongoing work~\cite{JaKhMe25+}, a similar method is used to prove the analogous claim for products of locally compact unimodular groups.

\subsection{Ideas, Sketch of the Proof, and Further Results}
\label{intro:sketch}

Throughout the paper, we let $G''=G\times G'$ be the product of two infinite countable groups. 
Let $o$ and $o'$ be the neutral elements of $G$ and $G'$ respectively, and $o'':=(o,o')$. 
As mentioned, we first prove the claim when two groups are finitely generated. Equip each of $G$ and $G'$ with a Cayley graph and let $d$ and $d'$ be the resulting graph-distance metrics. 

\subsubsection{Weak Containment and Weak Factor}
\label{intro:weakFactor}

The cost of a stationary random marking of $G$, or more generally, a probability-measure-preserving (pmp) action of $G$, is defined similarly, by requiring the graph to be a factor of the random marking/action, where \textit{factor} means a measurable $G$-equivariant function (without additional randomness). 
Such a factor graph is called a \defstyle{graphing} of the random marking/action. In fact, this definition requires that the random marking/action is essentially free; i.e., almost every sample has a trivial stabilizer (otherwise, this definition would be different from the original definition of cost).

The following fundamental results are important for studying the maximum cost of group actions.
Kechris proved in~\cite{bookKe10} that, for a finitely generated group $\Gamma$, the cost of essentially free $\Gamma$-actions is monotone under \textbf{weak containment}; see \Cref{ap:weakFactor} for the definition.
In~\cite{AbWe11}, it is proved that every iid marking of $\Gamma$ is weakly contained in any essentially free $\Gamma$-action, and hence, the iid markings have the maximum possible cost among essentially free $\Gamma$-actions. 

Inspired by the above results, the notion of \textbf{weak factor} (i.e., the weak limit of a sequence of factors) is defined in~\cite{AbMe22} for point processes on locally compact groups (for simplicity, we assume in this paragraph that $\Gamma$ is non-discrete and compactly generated). In~\cite{AbMe22}, a cost monotonicity result is proved for weak factors as well. It is also proved that every essentially free pmp action is \textit{isomorphic} to a point process, and that the Poisson point process is a weak factor of every essentially free action. Hence, the Poisson point process attains the maximum cost. 
This has proved to be a powerful tool in the theory of cost. In particular, several papers prove fixed price 1 for some classes of groups by  constructing a (marked) point process which is a weak factor of Poisson and has cost 1; see e.g., \cite{AbMe22,FMW,Me23}, discussed further below.

In this paper, for countable groups, we deduce from Theorem~5.10 of~\cite{AbMe22} the following slightly generalized cost-monotonicity,\footnote{When the space is compact (or locally compact), this result was known by the experts and we do not claim originality. In this case, the second claim of the lemma is implied by Theorem~2.25 of~\cite{FMW}. The case of non-locally-compact spaces seems to be new. 
} 
which covers arbitrary continuous pmp actions (not just point processes), proved in \Cref{ap:weakFactor}:
\begin{lemma}[Cost Monotonicity]
	\label{lem:monotonicity}
	Assume $\Gamma$ is finitely generated and $\alpha$ and $\alpha'$ are pmp actions of $\Gamma$. If $\alpha$ is continuous and is a weak factor of $\alpha'$, then $\cost(\alpha)\geq \cost(\alpha')$. Therefore, if one also has $\cost(\alpha)=1$ and $\alpha'$ is an iid marking, then $\Gamma$ has fixed price 1.
\end{lemma}
We also provide an alternate proof of this lemma based on the following result, which seems to be new. 
In fact, this also gives an alternate proof for the discrete case of the cost-monotonicity result of~\cite{AbMe22}.

\begin{lemma}
	\label{lem:WFvsWC}
	Let $\alpha:\Gamma\curvearrowright(X,\mu)$ and $\alpha':\Gamma\curvearrowright(X',\mu')$ be pmp actions. 
	\begin{enumerate}[label=(\roman*)]
		\item \label{lem:WFvsWC:1} If $\alpha$ is a continuous action and is a weak factor of $\alpha'$, then $\alpha$ is weakly contained in $\alpha'$.
		\item \label{lem:WFvsWC:2} If $X$ is compact and $\alpha$ is weakly contained in $\alpha'$, then $\alpha$ is a weak factor of $\alpha'$. In the noncompact case, an extension of $\alpha$ to a metrizable compactification of $X$ is a weak factor of $\alpha'$.
	\end{enumerate}
\end{lemma}

\subsubsection{Motivation: Vertically-Constant Markings}
\label{intro:motivation}

Gaboriau's problem was motivated by having shown that $G''$ admits an essentially free action with cost 1 (\cite{Ga00cost}). 
An example of such an action is the product of an essentially free action of $G$ and one of $G''$. To use the ideas later, we sketch a simple proof for the special case $(\bs m, \bs m'')$, where $\bs m$ is an iid marking of $G$ and $\bs m''$ is an iid marking of $G''$ independent from $\bs m$.
For $x\in G$, let $\tau(x)$ be the closest element of $G$ to $x$ that satisfies $\bs m(\tau(x))>\bs m(x)$ (if there are ties, choose the one with the smallest value of $\bs m$). Connect every point $(x,x')\in G''$ to $(\tau(x),x')$. This creates a forest on $G''$ with horizontal edges such that every point has exactly one outgoing edge. Hence, the expected degree of $o''$ is 2 (by the \textit{mass transport principle}). Also, every connected component is infinite.
Then, a low-cost connected graphing of $G''$ is obtained using the following \defstyle{infinite semi-touching technique}:\footnote{This technique can be used for factor graphs of other actions as well. In fact, for essentially free actions, it is possible to avoid extra randomness in~\eqref{eq:infiniteTouching} for merging the components, but this is not needed for the purpose of this paper. See Lemma~I.14 of~\cite{Ga00cost} or the technique of \textit{commuting relations} in~\cite{Ga24}.}
\begin{equation}
	\label{eq:infiniteTouching}
	\parbox{0.82\linewidth}{
		\emph{If a disconnected factor graph is given and two of its connected components have the infinite semi-touching property; i.e., they contain infinite sequences $(x_i)_i$ and $(y_i)_i$ of distinct points respectively such that $\sup_i d(x_i,y_i)<\infty$, then the two components are merged a.s. after adding a small bond percolation.}
	}
\end{equation}
The last bond percolation can be added as a factor of $\bs m''$, which completes the proof. 
Note that infinite semi-touching is weaker than \textit{infinite touching}; i.e., the existence of infinitely many edges between the two components, which is proved for the cells of the tessellation in~\cite{FMW}.

An iid marking of $G$ can also be regarded as a marking of $G''$ that depends only on the first coordinate in an iid manner.  The latter is called a \defstyle{vertically-constant iid marking} in this paper. 
According to the above claim, it is natural to ask whether the vertically-constant iid marking is a weak factor of iid or not. If it were, then the claim of \Cref{thm:fixedpriceFG} would follow. But the next two propositions show that this holds only for amenable $G'$ (which is already known to satisfy Gaboriau's problem).

\begin{proposition}
	\label{prop:amenable}
	If $G'$ is amenable, then the vertically-constant iid marking, defined above, is a weak factor of iid
\end{proposition}
The proof uses the \textit{propagation method} of~\cite{AbMe22,Me23} (see Proposition~6.6 of~\cite{AbMe22} and Theorem~3.1 of~\cite{Me23}). We only provide an sketch since we will not use this result directly.

\begin{proof}[Sketch of the proof]
	Let $F'_1,F'_1,\ldots$ be a F{\o}lner sequence in $G'$. For every $n$, let $\Phi_n$ be a Bernoulli point process on $G''$ with intensity $\epsilon_n:=1/\norm{F'_n}$ (i.e., keep the points randomly with probability $\epsilon_n$, and delete them otherwise, independently). 
	Consider the collection $\{(x'' F'_n, \bs m''(x'')): x''\in \Phi_n\}$ of marked copies of $F'_n$, where $\bs m''$ is an iid marking of $G''$. One can prove that the weak limit of this process is a (Poisson) random collection of \textit{vertical sections} of the form $\{x\}\times G'$, where each section carries a single random mark. It is then straightforward to obtain a vertically-constant iid marking using a Voronoi tessellation in $G$.
	%
\end{proof} 

The nonamenable case is refuted by Russell Lyons\footnote{Personal communication in December 2025. This was posed as a problem in earlier versions of the paper.
} as follows:
\begin{proposition}
	\label{prop:nonamenable}
	If $G'$ is nonamenable and finitely generated, then every vertically-constant random marking of $G''$ that is a weak factor of iid, is essentially constant.
\end{proposition}
\begin{proof}
	If there is a counterexample, one may apply a $\{0,1\}$-valued function to the marks and obtain a counterexample with mark space $\{0,1\}$. 
	So, it is enough to prove the claim when the mark space is $\{0,1\}$. Let $\bs m''$ be an iid marking of $G''$ with mark space $[0,1]$ and $\bs n''$ be any factor of $\bs m''$ with mark space $\{0,1\}$. 
	Let $A$ be the event $\bs n''(o'')=1$ and $A'$ be the event that $\bs n''((o,x'))=1$ for some $x'$ adjacent to $o'$ (in some fixed Cayley graph of $G'$).
	Let $q:=\myprob{A}$ and $q':=\myprob{A'}$.
	The random marking $\bs n''((o,\cdot))$ of $G'$ is also a factor of the iid marking $x'\mapsto \bs m''(\cdot,x')$ of $G'$ (the latter has mark space ${[0,1]}^G$). Therefore, Lemma~2.3 of~\cite{LyNa11} implies that $q'\geq q/(q+\rho(1-q))$, where $\rho$ is the spectral radius of the simple random walk on $G'$, and $\rho<1$ by nonamenability. It follows that, if $q$ is away from 0 and 1, then $\myprob{A'\setminus A}\geq q'-q>0$ is also away from 0. Since $A'\setminus A$ is clopen in ${\{0,1\}}^{G''}$, this implies that a sequence of such $\bs n''$ cannot converge weakly to a vertically-constant marking except when $q$ converges to 0 or 1, which implies that the limit is essentially constant.
\end{proof}

\subsubsection{A Poisson horoball process}
\label{intro:poisson}

\begin{figure}[t]
	\begin{center}
		\includegraphics[width=.49\textwidth]{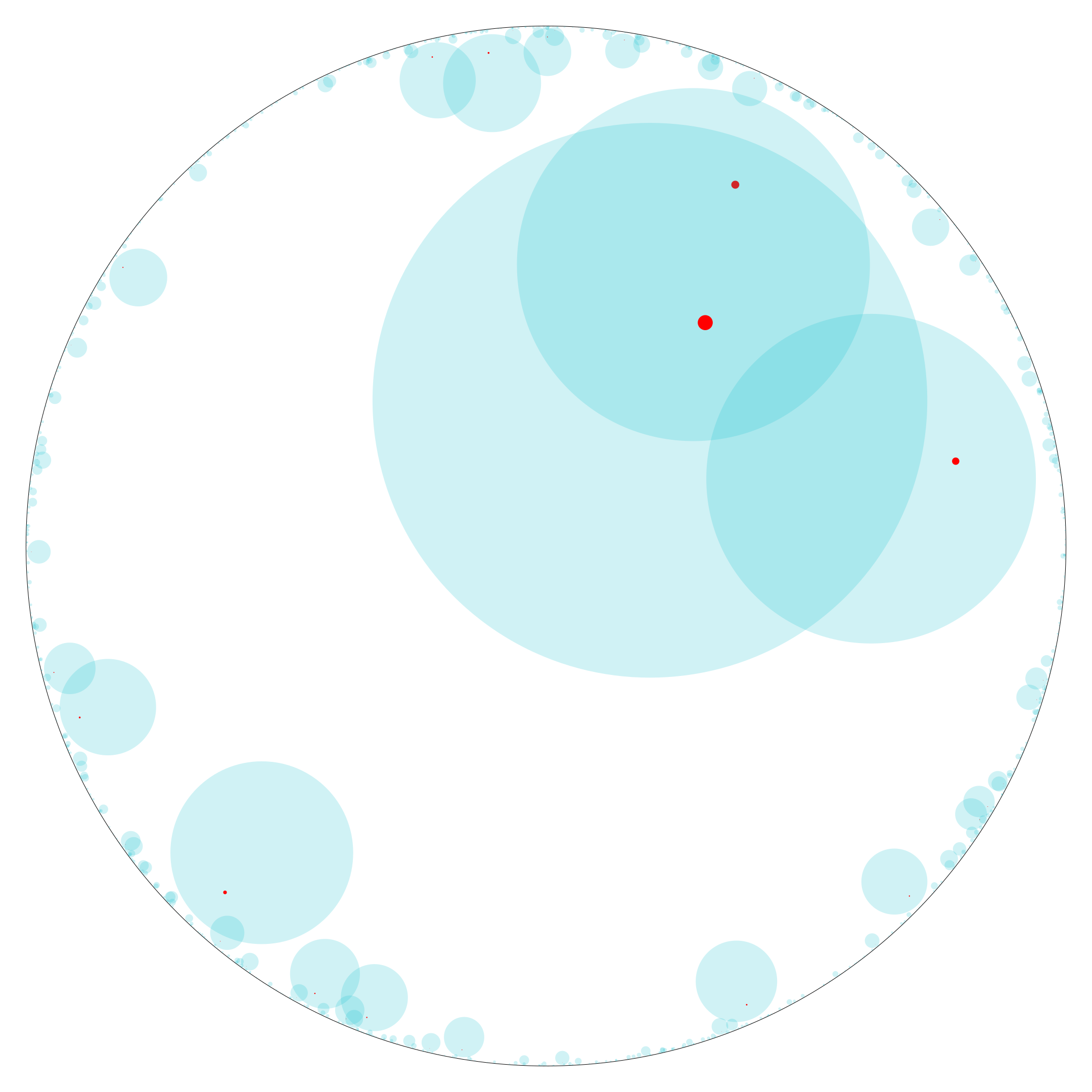}
		\includegraphics[width=.49\textwidth]{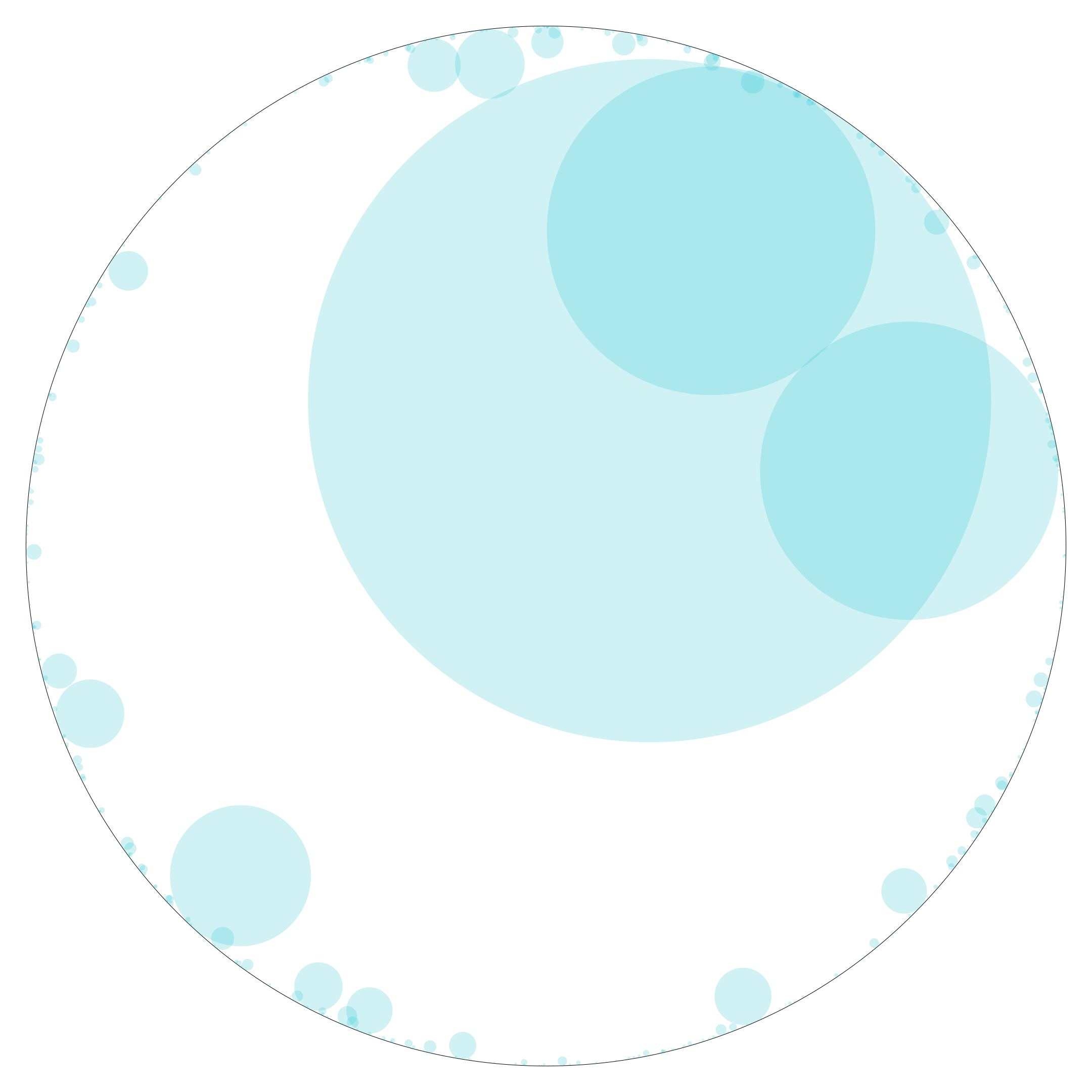}
		\caption{In the Poincar\'e model of the hyperbolic plane, a Poisson random collection of balls of fixed radius (see~\Cref{intro:poisson}) is shown on the left, and a Poisson horoball process is shown on the right. The centers of the balls are shown in red.}
		\label{fig:hyperbolic}
	\end{center}
\end{figure}

In the nonamenable case, we start by trying a propagation method similar to the one in the proof of \Cref{prop:amenable}. By \Cref{prop:nonamenable}, one cannot hope to obtain a vertically-constant iid marking in the limit. Nevertheless, we will show that the limiting process (a \textit{Poisson horoball process}) is a weak factor of iid and can be used for constructing a low-cost graphing, if $G$ and $G'$ have a nice growth behavior (mentioned later in~\eqref{eq:ratio}) that will be described in the next subsections. We will then show in \Cref{intro:perturb} how to remove the growth condition. Hence, \Cref{lem:monotonicity} implies fixed price 1.

Consider the following propagation method:
Choose a Bernoulli process on $G''$ with a small parameter $\epsilon$; i.e., keep every point with probability $\epsilon$. Then, for a suitable metric on $G''$ that will be described below, and for every point $x''$ of the Bernoulli process, put a large ball in $G''$ centered at $x''$ whose volume is proportional to $\epsilon^{-1}$. This results in a random collection of balls (see \Cref{subsec:pp} for the topological background). 
Equip the interior of each ball with a vertically-constant iid marking. As mentioned above, one cannot obtain a vertically-constant marking of $G''$ in the limit. 
Instead, each ball converges in a suitable sense to (roughly) a \textit{horoball}. Horoballs are described in \Cref{subsec:boundary}; imagine a large circle or diamond in the plane that converges to a half-plane. Then, one obtains a \textit{Poisson horoball process} by considering a subsequential limit (see \Cref{sec:proof} for details), where each horoball is equipped with a vertically-constant iid marking inside it.

In the above construction, the balls in $G''$ are defined using the following weighted~$L^1$ metric:
\begin{equation}
	\label{eq:rho}
	\rho_c((x,x'),(y,y')):=d(x,y)+ d'(x',y')/c,
\end{equation}
where $0<c<\infty$ will be determined later. The balls with the metric $\rho_c$ have a diamond-like shape.
As schematic images, the ball/horoball processes are shown in Figures~\ref{fig:hyperbolic} and~\ref{fig:tree} for the hyperbolic plane and the 3-regular tree respectively, although these spaces do not have a product structure (in the hyperbolic plane, the centers of the balls form a Poisson point process). For products, it is useful to imagine horoballs in the plane or $\mathbb Z^2$, see Figure~\ref{fig:type} in \Cref{subsec:boundary} (but the ball process in $\mathbb Z^2$ has a trivial limit). 

\subsubsection{A graphing on the union of the horoballs}
\label{intro:graphing}

\begin{figure}[t]
	\begin{center}
		\includegraphics[width=.49\textwidth]{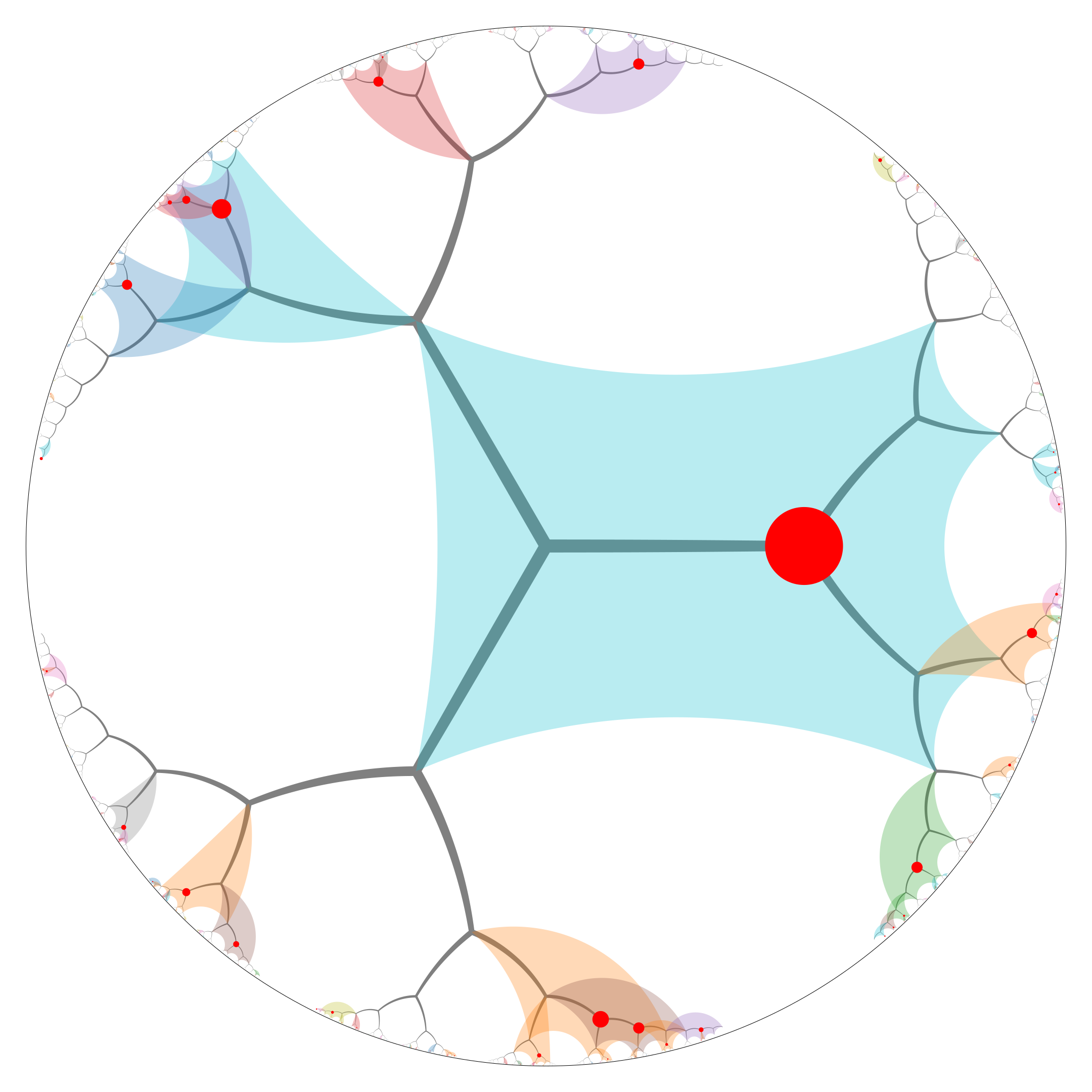}
		\includegraphics[width=.49\textwidth]{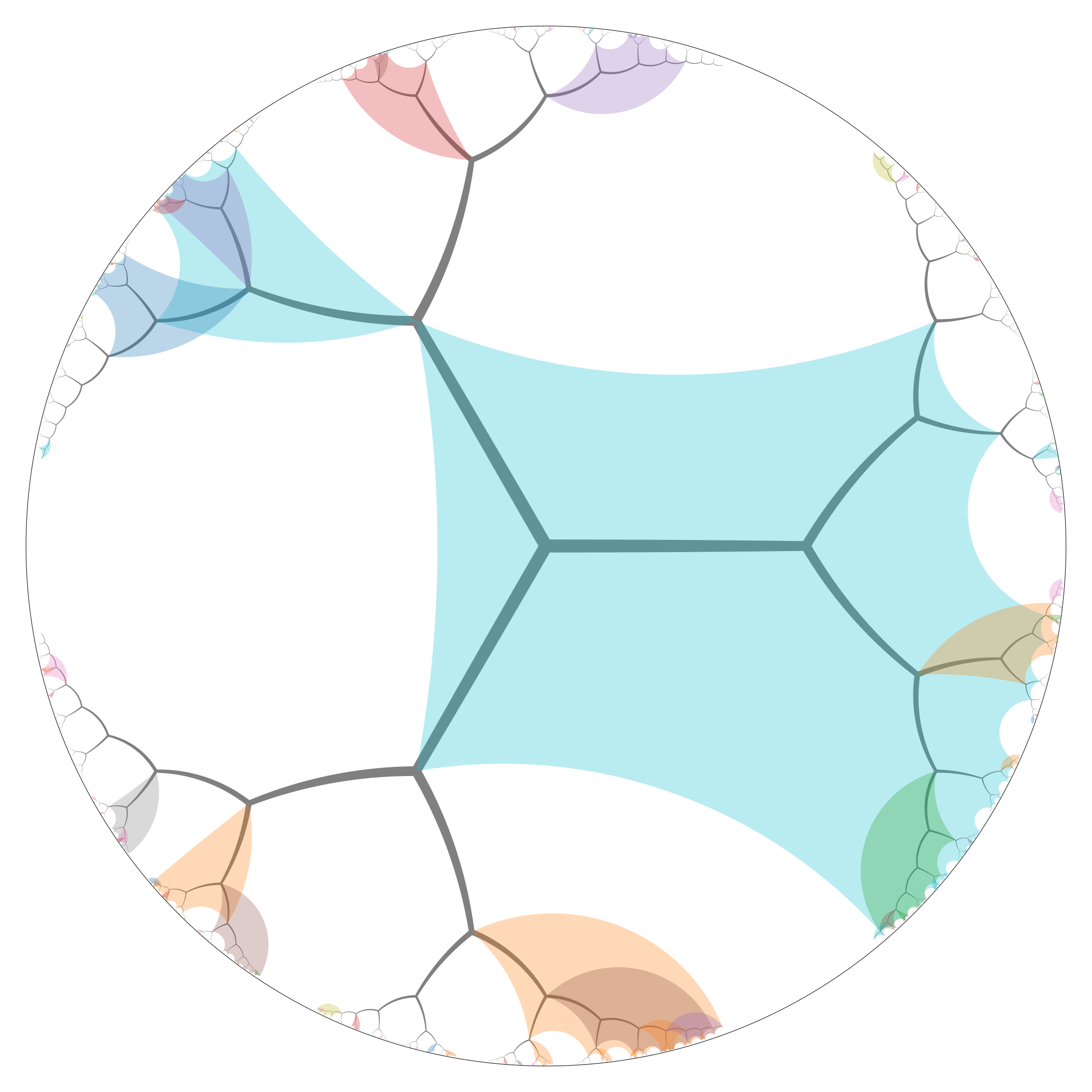}
		\caption{Models on the 3-regular tree, defined similarly to Figure~\ref{fig:hyperbolic}.
		}
		\label{fig:tree}
	\end{center}
\end{figure}

Given the Poisson horoball process, defined above, we explicitly construct a graph on the union of the horoballs  as follows. First, inside each horoball, we proceed similarly to the proof of \Cref{prop:amenable} by constructing a horizontal forest with exactly one outgoing edge for every vertex, and then, adding a small percolation; see \Cref{sec:proof} for how this is done (it turns out that this construction can be done without the vertically-constant iid marking inside the horoballs, and hence, this marking will not be defined in \Cref{sec:proof}). In order to use the infinite semi-touching technique, it is necessary that all horizontal and vertical sections of the horoballs are infinite almost surely. This will be proved by a certain volume growth condition~\eqref{eq:ratio} discussed later in \Cref{intro:growth}. We will also discuss in \Cref{intro:perturb} how to get rid of the growth condition.

Second, we will show that the horoballs in the Poisson horoball process have pairwise infinite semi-touching~\eqref{eq:infiniteTouching} (under the same growth condition), and hence, adding a small percolation connects all horoballs. This property will be proved elementarily by constructing two paths with \textit{equal slopes} $c$ in the two horoballs, that remain in bounded distance (see Figures~\ref{fig:type} and the figure in the proof of \Cref{lem:infiniteTouching}). So, a connected graph is obtained on the union of the horoballs. In the next subsection, we will then describe how to obtain a graphing on the whole $G''$.

We highlight that the above construction is elementary and does not rely on amenability nor on double-recurrence, which are used in~\cite{FMW}. We do not prove or disprove whether a typical horoball is amenable (if it were amenable, a cost 1 graphing inside the horoball would exist, which is the case for a typical cell in~\cite{FMW}). Also, the proof of infinite semi-touching does not rely on the notion of double-recurrence (nor mixing), but the construction of the paths with equal slopes are inspired from~\cite{FMW}.
See \Cref{intro:relatedWorks} for connections and comparison with~\cite{FMW} and the recent paper~\cite{BeBo25cost}.

\subsubsection{Induction lemma for multi-sets}
\label{intro:induction}

Having constructed a graphing on the union $\bs S$ of the horoballs, we now describe how to obtain a low-cost graphing on the whole $G''$.
In the above construction, it should be noted that the horoballs may overlap and do not cover $G''$. In the construction, we regard the multiple points as distinct points (which are still connected to each other by the constructed graphing). 
Since every vertex had exactly one outgoing edge before adding the percolation, the expected degree of a \textit{typical point} of $\bs S$ is arbitrarily close to 2 (this notion will be formalized in \Cref{sec:induction} using Palm theory).
If there were no multiple points, the \textit{induction lemma} would result in a low-cost graphing on the whole $G''$; see~\Cref{subsec:cost} and Equation~\eqref{eq:induction}. 
To manage the multiplicities, we prove the following version of the induction lemma for random \textit{multisets} in a group. More precisely, we use the notion of \textit{marked point processes} from stochastic geometry, which allows multiple points. 
In \Cref{sec:induction}, the notions of graphing and cost are extended to marked point processes, and it is proved that:
\begin{theorem}[Induction Formula for Marked Point Processes]
	\label{thm:induction}
	Let $\Gamma$ be a countable group.
	For every essentially free pmp action $\alpha:\Gamma\curvearrowright (X,\mu)$, and for every nonempty (simple) marked point process on $\Gamma$ that is a factor of $\alpha$ and has finite intensity, the same induction formula (see~\eqref{eq:induction}) is valid.
\end{theorem}

By this theorem, a low-cost graphing on $G''$ is constructed, as desired. Two proofs of \Cref{thm:induction} are given in \Cref{sec:induction}. The first proof reduces the claim to the induction lemma for countable Borel equivalence relations (the claim for group actions is not sufficient).
To be more self-contained, an elementary proof is also included, which mimics the proof of the original induction lemma.

\subsubsection{The condition of nice growth behavior}
\label{intro:growth}

A delicate challenge in the above proof is ensuring that only \textit{good} horoballs appear in the limiting horoball process. More precisely, we need two properties: The infinite semi-touching property of horoballs, and also we should be able to construct a low-cost graphing inside the horoballs. 
Under the metric $\rho_c$, defined in~\eqref{eq:rho}, the infinite semi-touching property does not hold for all horoballs, but holds for those horoballs that correspond to a pair of \textit{boundary points} of $G$ and $G'$ respectively, which we call \defstyle{horoballs of type II} (\Cref{def:type} and Figure~\ref{fig:type}). 
In addition, the horoballs of type II are precisely those horoballs that have infinite horizontal and vertical sections, as required in \Cref{intro:graphing}.
 
For ensuring that only horoballs of type II appear in the aforementioned Poisson horoball process, 
the following condition should be satisfied (see the proof of \Cref{lem:horoball}), where $v_n$ and $v'_n$ denote the volumes of the balls with radius $n$ in $G$ and $G'$ respectively, and $\oball{n}{o'',\rho_c}$ denotes a ball of radius $n$ under the metric $\rho_c$:
\begin{equation}
	\label{eq:ratio}
	\limsup_n \frac{\card{\oball{n}{o'',\rho_c}}}{\max\{v_n,v'_{\floor{cn}}\}}=\infty.
\end{equation}
Note that if this condition does not hold, then the middle vertical section or the middle horizontal section of a ball would occupy a non-negligible portion of the ball, which shows that the origin would be near a corner of a typical ball that contains it, and this would create horoballs not of type II.
For~\eqref{eq:ratio} to hold, one can see that $c$ should be equal to $\log a/\log a'$, where $a$ and $a'$ are the growth rates of $G$ and $G'$, but we don't know whether this is sufficient:

\begin{problem}
	\label{prob:ratio}
	Does~\eqref{eq:ratio} always hold if $c=\log a/\log a'\in(0,\infty)$?
\end{problem}

As an evidence, we prove:

\begin{proposition}
	\label{prop:growth}
	Assume $a>1$ and $a'>1$, and let $c:=\log a/\log a'$. If for some arithmetic sequences $(r_n)_n$ and $(r'_n)_n$ in $\mathbb N$, the limits $\lim_n v_{r_{n+1}}/v_{r_n}$ and $\lim_n v'_{r'_{n+1}}/v'_{r'_n}$ exist, then~\eqref{eq:ratio} holds.
\end{proposition}
This claim is proved by splitting a $\rho_c$-ball into \textit{vertical cylindrical slices}: 
\begin{equation}
	\label{eq:ball}
	\card{\oball{n}{o'',\rho_{c}}} = 	\sum_{t=0}^{n} s_{(n-t)}  v'_{\floor{c t}},
\end{equation}
where $s_n:=v_n-v_{n-1}$ is the volume of the sphere of radius $n$, and by showing that all of these slices have roughly equal volumes (if $r_n:=n$): 

\begin{proof}[Proof of \Cref{prop:growth}]
	We prove that $\lim_n {\card{\oball{n}{o'',\rho_c}}}/{v_n}=\infty$. We prove this only in the case $r_n=n$, and the claim for general arithmetic sequences can be proved similarly.
	Assume $b:=\lim v_{{n+1}}/v_{n}$ exists. 
	So, $\forall t\in\mathbb Z: \lim_n v_{{n-t}}/v_{n} = b^{-t}$. One can deduce that $b=\lim_n v_{n}^{1/n}=a$. Also, one has $\lim_n s_{(n-t)}/v_n = (1-\frac 1 a) a^{-t}>0$ (since $a>1$) and $v'_{\floor{c t}}\geq (a')^{\floor{c t}}\geq (a')^{-1} a^{t}$. 
	Therefore, $\liminf_n s_{(n-t)}  v'_{\floor{c t}} / v_n \geq (1-\frac 1 a) (a')^{-1}$ for all $t$. Hence, \eqref{eq:ball} implies that $\card{\oball{n}{o'',\rho_{c}}}/v_n$ converges to infinity and the claim is proved.
\end{proof}

The author is not aware of any group known to violate the assumption of \Cref{prop:growth}. However, since groups with oscillating behavior of $v_{n+1}/v_n$ are difficult to control, we give up with \Cref{prob:ratio} and proceed with another way, by modifying the proof from the beginning. 

\subsubsection{Removing the condition by using perturbed balls}
\label{intro:perturb}
To remove the growth condition~\eqref{eq:ratio} (or that of \Cref{prop:growth}), we modify the above proof by replacing $\rho_c$-balls (used in the propagation method) with another shape from the beginning. For this, it is essential to use a suitable notion of convergence. The latter is the convergence of \textit{point processes of pointed marked sets}, defined later in \Cref{subsec:pp}.

To perturb the balls, we change the radius of the vertical slices (see~\eqref{eq:ball}) from $\floor{ct}$ to another value, namely $f(t)$, and call the resulting shape a \textbf{perturbed diamond}. We will choose $f$ suitably, but the challenge is a trade-off between ensuring that the slices have roughly the same volume, and that the \textit{slope} of the boundary of the perturbed diamonds converges to $c$. So, we cannot achieve both if~\eqref{eq:ratio} fails. Instead, the idea is that not all slices need to have the same volume: It is enough that only a few of the slices have volume comparable to $v_n$, and that the number of those \textit{good} slices converges to infinity (possibly very slowly). This way, we can put the {good} slices with sufficiently far distance from each other, and carefully dampen the deviations of the slope from $c$, such that the slope converges to $c$ in the limit (\Cref{lem:linear}). Then, we will show that the perturbed diamonds converge (roughly) to horoballs with slope $c$ (\Cref{lem:perturbedHoroball}), and no bad horoballs appear in the resulting horoball process a.s. (\Cref{lem:perturbedHoroball2,lem:perturbedHoroball3,lem:horoball}). This way, the general finitely generated case of \Cref{thm:fixedpriceGeneral} is proved without any assumptions.

\unwritten{We expect that the $l_2$ metric (which is the metric chosen in~\cite{FMW}) can also be used and yields the same Poisson horoball process, but requiring further assumptions.}

\subsection{Related Works}
\label{intro:relatedWorks}

The paper~\cite{FMW} proves the fixed price 1 property in two settings: Products of the automorphism groups of at least two regular trees, and also higher rank semisimple real Lie groups (these are not discrete groups, but the notions of cost and fixed price are already defined for them; see~\cite{AbMe22}). Their proof leverages the weak subsequential limits of low-intensity Poisson-Voronoi tessellations, called the \textit{ideal Poisson-Voronoi tessellation} (IPVT), which has been introduced and developed recently in~\cite{sandeep19,BuCuPe25cheeger,DCE+23}. They prove the remarkable property that, in the two settings mentioned above, every pair of cells of the IPVT touch each other at infinitely many places. This is proved by using double-recurrence of the action of the group on the \textit{corona space} (which is equivalent to the space of horoballs). Also, they consider a cost 1 graphing inside the cells by proving that a typical cell is amenable. 

By extending the ideas of~\cite{FMW} and the propagation method, the paper~\cite{Me23} proves the fixed price 1 property in a more general setting, and in particular, for another class of products of groups. 

Shortly after the present paper was published on arXiv, L. Bowen and E. Bevilacqua published a similar result in~\cite{BeBo25cost}, under a certain growth condition, which has been prepared independently in parallel to this work. 
They also use a Poisson horoball process, but their proof method is different from ours. They extend the ideas of~\cite{FMW} and prove that a typical horoball  in an \textit{exact} group is amenable 
(by introducing and developing the notions of cost and \textit{limit-amenability} for \textit{infinite-measure-preserving (imp)} actions).
They also show the infinite semi-touching property of horoballs by proving double-recurrence.

The following are the connections and differences of the present work with the ones mentioned above:
\begin{itemize}
	\item The present paper uses the propagation method instead of the IPVT. The connection with~\cite{FMW} is that, when non-perturbed diamonds are used, the resulting Poisson horoball process is equivalent (as an action of $G''$) to the IPVT (by considering the Voronoi diagram of the set of centers of the horoballs). 
	However, an important difference is that the present paper uses a different notion of convergence: In~\cite{FMW,BeBo25cost}, only the convergence of the set of centers of the balls are considered, but in the present work, the convergence is in the sense of random collections of (pointed) subsets of the group (a notion from stochastic geometry described in \Cref{subsec:pp}). This is what allows us to perturb the balls and obtain a different limit if the volume growth condition~\eqref{eq:ratio} fails, although we do not have a counter example for~\eqref{eq:ratio} at this moment. 
		
	\item
	Another difference with~\cite{FMW,BeBo25cost,Me23} is the proofs of the following two claims, in a significantly simpler way than the analogous results in these papers: The construction of a low-cost graphings inside the horoballs, and also the infinite semi-touching property of horoballs. We prove no amenability, mixing or double-recurrence properties. Instead, we provide much simpler direct proofs (the proof of the second claim is inspired from~\cite{FMW}, as discussed in \Cref{intro:graphing}).
	
	\item
	The idea of using overlapping sets instead of tessellations and the induction formula for multisets (\Cref{thm:induction}) seem to be new.
	
	\item
	It should be noted that~\cite{FMW} considers the $l_2$ product metric, but one can show that the resulting horoball process is identical to the one with the metric $\rho_c$ (imagine a large circle converging to a half-plane). We expect this property to hold for general products under only a growth condition.
\end{itemize}

\subsection{The Structure of the Paper}

The basic definitions and properties are provided in \Cref{sec:def}, including the notion of cost, horoballs, and point processes of horoballs. In particular, two types of horoballs on $G''=G\times G'$ are described in \Cref{def:type}. \Cref{sec:diamond} defines perturbed diamonds and the fine tuning of the perturbations. It also provides criteria for the convergence of perturbed diamonds to (slightly perturbed) horoballs, and similar criteria for point processes of perturbed diamonds. The extension of the induction formula to random multisets is done in \Cref{sec:induction} .
Finally, the proof of the theorem is provided in \Cref{sec:proof}.

\section{Definitions and Notation}
\label{sec:def}

\subsection{Notation}
\label{subsec:notation}

We will mostly use boldface symbols like $\bs S, \bs C, \ldots$ to refer to random objects.
The cardinality of a set $A$ is denoted by $\norm{A}$.
If $\rho$ is a metric on a set $M$, $x\in M$ and $r\geq 0$, then $\oball{r}{x}:=\oball{r}{x,\rho}$ denotes the closed ball of radius $r$ in $M$ centered at $x$. If $H$ is an undirected graph (or a set of edges) and $x$ is a vertex of $H$, $\deg(x,H)$ denotes the degree of $x$ in $H$.

As mentioned in \Cref{intro:sketch}, we fix finitely generated groups $G$ and $G'$  with neutral elements $o$ and $o'$ respectively. Let $G'':=G\times G'$ and $o'':=(o,o')$. Equip $G$ and $G'$ with arbitrary Cayley graphs, and let $d$ and $d'$ be the resulting graph-distance metrics. 
Let $v_n:=\card{\oball{n}{o,d}}$ and $v'_n:=\card{\oball{n}{o',d'}}$ denote the volumes (i.e., the number of points) of the balls of radius $n$ in $G$ and $G'$, respectively.
Using the fact $v_{m+n}\leq v_m v_n$, one gets that $v_n^{1/n}$ is non-increasing, and hence, converges to some constant $a\geq 1$, which is called the \defstyle{growth rate} of $G$. Let $a'$ be the growth rate of $G'$. We assume that $a>1$ and $a'>1$ (otherwise, one of the groups is amenable and the claim of \Cref{thm:fixedpriceGeneral} is already known). Also, we always equip $G''$ with the weighted $l_1$ metric $\rho_c$ defined in~\eqref{eq:rho}, where $c=\log a/\log a'$.


Throughout the paper, we use unprimed, primed or double-primed symbols for objects that refer to $G$, $G'$ or $G''$ respectively.

If $M$ is a countable set and $p:M\times M\to [0,1]$ is a symmetric function, the \defstyle{bond percolation} with intensity measure $p$ (on the complete graph) is a random subset $\Phi$ of $M\times M$ defined as follows: Put every unordered pair $\{x,y\}$ in $\Phi$ with probability $p(x,y)$, independently from all other pairs. A pair is called \defstyle{open} if it is in $\Phi$ and \defstyle{closed} otherwise.

We need the notions of vague topology and the Fell topology (see, e.g., \cite{bookScWe08}).
Let $E$ be a locally compact second countable Hausdorff space. The \defstyle{vague topology} is defined on the set of locally finite Borel measures on $E$ as follows: $\mu_n\to\mu$ when $\int fd\mu_n\to\int fd\mu$ for all compactly-supported continuous functions $f:E\to\mathbb R$.
The space of such measures is Polish. In particular, the space of integer-valued measures on $E$ is also Polish. Note that an integer-valued measure on $E$ corresponds to a multi-set in $E$. A \defstyle{point process} in $E$ is a random discrete multi-set in $E$. It is called \defstyle{simple} if it contains no multiple points a.s.

Also, the \defstyle{Fell topology} is defined on the set of closed subsets of $E$, and makes it compact. 
Here, we need the case where $E$ is discrete. In this case, the Fell topology can be defined as follows: $B_n\to B$ when $\forall x\in H: \identity{B_n}(x)\to\identity{B}(x)$.

\subsection{Cost}
\label{subsec:cost}

The notion of cost for countable group actions is a special case of the analogous notion for measured countable Borel equivalence relations (CBERs), but we will try to avoid CBERs in defining cost and use a more probabilistic language.

Let $\Gamma$ be a countable group with neutral element $o$. By convention, all actions in this paper are Borel actions on some complete separable metric space.\footnote{In fact, only the Borel structure of the underlying space is needed, except in weak factors which require (only) the topology as well.}
Consider such an action of $\Gamma$ on a space $E$. If $\mu$ is a Borel measure on $E$, then the action is called \defstyle{probability-measure-preserving (pmp)} if $\mu$ is a probability measure that is preserved under the action of every element of $\Gamma$. 
Examples include:
\begin{itemize}
	\item A \defstyle{stationary random marking} of $\Gamma$ with marks in a Polish space $\Xi$ (or in other words, a stationary stochastic process indexed by $\Gamma$). Equivalently, a probability measure on $\Xi^\Gamma$ that is is invariant under left multiplication by every element of $\Gamma$.
	Special cases include \defstyle{iid markings} of the points of $\Gamma$ (we usually take $\Xi=[0,1]$), and \defstyle{stationary random subsets} of $\Gamma$ (where $\Xi=\{0,1\}$).
	\item A \defstyle{stationary random graph} $\Pi$ on $\Gamma$; i.e., a probability measure on $\{0,1\}^{\Gamma\times\Gamma}$ (which is the distribution of $\Pi$) that is is invariant under left multiplications. In particular, a bond percolation on $\Gamma$ whose intensity measure $p$ satisfies $p(xy,xz)=p(y,z)$, $\forall x,y,z$.
\end{itemize}

A pmp action $\alpha:\Gamma\curvearrowright (E,\mu)$ is called a \defstyle{factor} of another pmp action $\alpha':\Gamma\curvearrowright (E',\mu')$ if there exists a measure-preserving function $\varphi: E'\to E$ (allowing $\varphi$ to be undefined on a $\mu'$-null set) that commutes with the actions of $\Gamma$ (i.e., $\varphi(h x') = h \varphi(x')$ for all $x'\in E$).  Also, $\alpha$ is called a \defstyle{weak factor} of $\alpha'$ if there exists a sequence $(\mu_n)_n$ of $\Gamma$-invariant probability measures on $E$ that converge weakly to $\mu$ such that $\Gamma\curvearrowright (E,\mu_n)$ is a factor of $\alpha'$ for every $n$. 

A pmp action $\alpha:\Gamma\curvearrowright (E,\mu)$ is \defstyle{essentially free} if, for $\mu$-a.e. $x\in E$, the stabilizer $\{h\in \Gamma: hx=x\}$ of $x$ is trivial. 
In this case, the \defstyle{cost} of $\alpha$ is 
\[
	\cost(\alpha):= \inf \frac 1 2 \omid{\deg(o,\Pi)},
\]
 where $\Pi$ is a stationary random graph on $\Gamma$ that is a factor of $\alpha$ and is connected a.s. ($\Pi$ is called a \defstyle{graphing} of $\alpha$). Heuristically, the cost is the \textit{infimal average number of edges per vertex} needed to connect all points of $\Gamma$ as a factor of the action.
The group $\Gamma$ has \defstyle{fixed price} if all essentially free actions of $\Gamma$ have the same cost. 

A useful tool in working with cost is the induction formula (Proposition~II.6 of~\cite{Ga00cost}), which we restate here in the special case of essentially free group actions (the claim also holds for the more general setting of measured countable Borel equivalence relations). Fix an essentially free pmp action $\alpha$ of $\Gamma$. Let $\bs S$ be a stationary random subset of $\Gamma$ that is a factor of $\alpha$ and is nonempty a.s. 
Let $\lambda(\bs S):=\myprob{o\in \bs S}$ denote the \defstyle{intensity} of $\bs S$.
Define the \defstyle{induced cost} of $\bs S$, given $\alpha$, by $\cost_{\alpha }(\bs S):=\inf \frac 1 2 \omidCond{\deg(o,\Pi)}{o\in\Pi}$, where the infimum is over all stationary random graphs $\Pi$ on $\Gamma$, as a factor of $\alpha$, such that $\restrict{\Pi}{\bs S}$ is  connected almost surely. In particular, $\cost_{\alpha }(\Gamma)=\cost(\alpha)$. The \defstyle{induction formula} states that
\begin{equation}
	\label{eq:induction}
	\cost(\alpha)-1 = \lambda(\bs S)\left(\cost_{\alpha}(\bs S)-1\right).
\end{equation}

We will need a version of the induction formula for multisets of $\Gamma$ (\Cref{thm:induction}), which will be provided in \Cref{sec:induction}.

\subsection{Boundary and Horoballs}
\label{subsec:boundary}


In this subsection, we recall the notions of boundary and horoballs from~\cite{Gr81hyperbolic} leveraging the notations of~\cite{DCE+23}. Since we will deal only with graphs, we provide the definitions only for this case, which is simpler to state.

Fix an infinite countable set $H$ and an \defstyle{origin} $o\in H$. Let $d$ be a boundedly-finite metric on $H$; i.e., a metric such that every ball in $H$ is finite (e.g., the graph-distance metric if $H$ is a graph, or the metric $\rho_c$ if $H$ has a product form). For every $x\in H$, consider the shifted distance function $d_x(\cdot):=d(x,\cdot)-d(x,o)$. By identifying every $x\in H$ with the function $d_x$, the \defstyle{horocompactification} $\overline H$ of $H$ is the closure of the set of shifted distance functions in the set of 1-Lipschitz functions on $H$ that vanish on $o$ (under pointwise convergence). The \defstyle{horoboundary} of $H$ is $\partial H:=\overline{H}\setminus H$ and its elements are called \defstyle{horofunctions}. Since $H$ is infinite, $\partial H$ is nonempty.

The following notation of~\cite{DCE+23} is helpful: A \textit{point of $\partial H$} is usually denoted by $\theta$ (or similar symbols) and the corresponding function on $H$ is denoted by $d_{\theta}$. In fact, $d_{\theta}$ and $\theta$ are the same objects, but viewed in two different ways. Note that $d_{\theta}(o)=0$. 

Given $\theta\in\partial H$ and $\delta\in\mathbb R$, the set $HB(\theta,\delta):= \{x\in H: d_{\theta}(x)\leq \delta \}$ is called a \defstyle{horoball} with center $\theta$ and \textbf{delay} $\delta$. The pair $(HB(\theta,\delta),\theta)$ is called a \defstyle{pointed horoball}. 
One might also call $HB(\theta,\infty):=H$ a \defstyle{horoball with infinite delay} and $(H,\theta)$ a pointed horoball with infinite delay.

\unwritten{Consider the case where $d$ is integer-valued (e.g., the graph-distance metric). This implies that, given a sequence of balls $B_{r_n}(x_n)$ in $H$, if $x_n\to \theta\in\partial H$ and $\delta_n:=r_n-d(x_n,o)\to \delta\in\mathbb Z$, then $B_{r_n}(x_n)$ converges to $HB(\theta,\delta)$ in the Fell topology; equivalently, $\forall y\in H: \identity{\oball{r_n}{x_n}}(y)\to \identity{HB(\theta,\delta)}(y)$. 
%
Similarly, if $\delta_n$ converges to $\infty$ (resp. $-\infty$), then $\oball{r_n}{x_n}$ converges to $H$ (resp. $\emptyset$). So, the set of balls and horoballs, together with $H$ and $\emptyset$, is compact. But if $d$ is not integer-valued, the mentioned properties do not hold. Instead, a subsequence of $\oball{r_n}{x_n}$ converges to a set that is sandwiched between $HB(\theta,\delta)$ and the \textit{open horoball} $ \{x\in H: d_{\theta}(x)< \delta \}$.}


For general groups $\Gamma$, we do not have many information about how the boundary and the horoballs look like, because they depend heavily on the geometry of the group (and on the choice of its generating set). It is even not clear if the set of points of a horoball uniquely determines its center (this is why \textit{pointed horoballs} will be considered in the next subsection). 
We also do not know whether every horoball is connected, or whether every point of $\partial \Gamma$ is the limit of a geodesic (or even a path) under the topology of $\bar\Gamma$.

We will only need the following two properties of the horoboundary: \Cref{lem:geodesic} states that the horoballs of a graph-distance metric are \textit{star-like} in some sense. Also, \Cref{lem:product-boundary} describes the horoboundary of $G''=G\times G'$ under the metric $\rho_c$, in terms of those of $G$ and $G'$.

\begin{lemma}
	\label{lem:geodesic}
	If $H$ is a graph equipped with the graph-distance metric, then for every $\theta\in\partial H$ and every $x\in H$, there exists an infinite path $(\gamma_i)_{i\geq 0}$ such that $\gamma_0=x$ and $d_{\theta}(\gamma_i)=d_{\theta}(x)-i$. 
\end{lemma}
It should be noted that $\gamma$ does not necessarily converge to $\theta$ in $\overline{H}$.
\begin{proof}
	Choose $x_n\in H$ such that $x_n\to \theta$. Let $\gamma^{(n)}$ be a geodesic from $x$ to $x_n$, and take a subsequential pointwise limit of $\gamma^{(n)}$ as $n\to\infty$.
\end{proof}

\begin{figure}[t]
	\begin{center}
		\includegraphics[width=.4\textwidth]{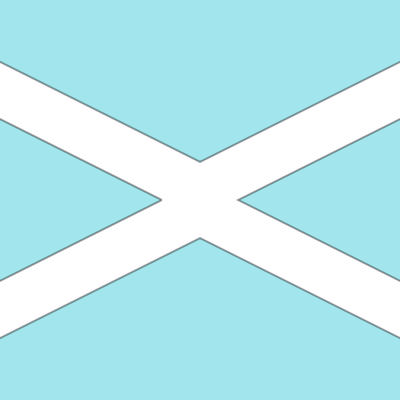}
		\quad\quad
		\includegraphics[width=.4\textwidth]{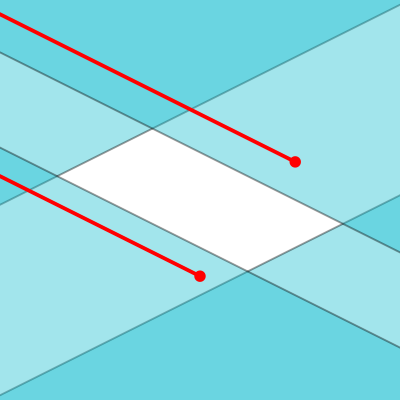}
		\caption{Horoballs in the plane with the metric $\rho_c$ \eqref{eq:rho}, where $c=\frac 12$, described at the end of \Cref{subsec:boundary}. On the left, the horoballs of type I are shown, which are corners in the plane. Horoballs of type II are shown on the right, and are half-planes with slope $\pm c$. Two paths with slope $c$ are also depicted, that are in bounded distance from each other.}
		\label{fig:type}
	\end{center}
\end{figure}


\begin{lemma}[Boundary of Products]
	\label{lem:product-boundary}
	Given the weighted $l_1$ metric $\rho_c$ on $G''$, defined in~\eqref{eq:rho}, one has
	\[
		\overline{G''} \equiv \overline{G}\times \overline{G'}.
	\]
	More specifically, the functions 
	\begin{equation}
		\label{eq:d_theta''}
		d_{(u,u')}(\cdot,\cdot):=d_{u}(\cdot)+d_{u'}(\cdot)/c,
	\end{equation}
	defined for $u\in\overline{G}$ and $u'\in\overline{G'}$, form all points of $\overline{G''}$.
\end{lemma}
\begin{proof}
	It can be seen that, if $(x_n,x'_n)$ is a sequence in $G''$, then it has a subsequence that converges to one of the functions $d_{(\theta,\theta')}$ mentioned in the lemma (consider 4 cases: whether $(x_n)_n$ and $(x'_n)_n$ escape to infinity or not). This implies the claim.
\end{proof}
Based on this lemma, we can define two types of boundary points and horoballs in $G''$:
\begin{definition}[Type of Horoballs]
	\label{def:type}
	A point $(u,u')\in \partial G''$, or a horoball centered at $(u,u')$, is:
	\begin{itemize}
		\item of \defstyle{type I} if either $u\in G$ and $u'\in \partial G'$, or $u\in \partial G$ and $u'\in G'$,
		\item of \defstyle{type II} if $u\in \partial G$ and $u'\in \partial G'$.
	\end{itemize}
\end{definition}

As a schematic example, \Cref{fig:type} shows some horoballs in the plane with the metric $\rho_c$, where $c=\frac 1 2$. Observe that every pair of horoballs of type II have the {infinite semi-touching} property~\eqref{eq:infiniteTouching} (this will be proved for general products in \Cref{lem:infiniteTouching}). For instance, for two of the horoballs in the figure, two paths are shown that are in bounded distance from each other and lie entirely in the two horoballs. Also, in horoballs of type II, all horizontal or vertical sections are infinite. These properties do not hold for horoballs of type I.

\subsection{Point Processes of (Marked) Balls/Horoballs}
\label{subsec:pp}

The notions of \textit{particle processes} and \textit{point processes of closed subsets} are central notions in stochastic geometry; see e.g., Chapter~3 of~\cite{bookScWe08}. We define similar notions in what follows to define point processes of horoballs, marked horoballs, or similar objects.

Let $H$ be an infinite countable set  
equipped with a boundedly-finite metric $d$. 
Let $\mathcal C(H)$ be the space of all pointed sets $(B,\theta)$, where  $B\subseteq H$ is a nonempty subset and $\theta\in\overline{H}$ (note that $\theta$ is not necessarily in $B$). One might equip $\mathcal C(H)$ with the product of the Fell topology and the natural topology of $\overline{H}$.
Noting that the empty set was excluded, one obtains that $\mathcal C(H)$ is a locally compact second countable metrizable space. Also,  one can see that a subset $K$ of $\mathcal C(H)$ is precompact if and only if there exists a finite set $F\subseteq H$ such that $\forall (B,\theta)\in K: B\cap F\neq\emptyset$.

Let $\bs C$ be a point process in $\mathcal C(H)$; i.e., a random discrete (multi-) set in $\mathcal C(H)$. By the previous statement, discreteness of $\bs C$ means that, for every $x\in H$, there are at most finitely many elements $(B,\theta)\in\bs C$ such that $x\in B$.
If $\bs C$ consists of only pointed balls/horoballs a.s., then it is called a \defstyle{point process of pointed balls/horoballs}. 
In fact, we will use \textit{perturbed} horoballs, defined later.

Additionally, we will need \defstyle{pointed marked balls/horoballs}. The latter are tuples of the form $(B,\theta; m)$, where $(B,\theta)$ is a pointed ball/horoball and $m:B\to \Xi$ is a marking of $B$, given some compact metric space $\Xi$, called \textit{the mark space}; e.g., $\Xi=[0,1]$. 
More generally, let $\mathcal C'(H)$ be the space of all tuples $(B,\theta;m)$, where  $B\subseteq H$ is a nonempty subset, $\theta\in\overline{H}$ and $m:B\to \Xi$ is a marking of $B$. In fact, we will only need constant markings, i.e., when $m$ is a constant function on $B$. So, it is safe to replace $\mathcal C'(H)$ with $\mathcal C(H)\times \Xi$. 

Similarly to the last case, it can be seen that $\mathcal C'(H)$ is a locally compact second countable metrizable space. Also, a subset $K$ of $\mathcal C'(H)$ is precompact if and only if there exists a finite set $F\subseteq H$ such that $\forall (B,\theta;m)\in K: B\cap F\neq\emptyset$. 
Let $\bs C'$ be a point process in $\mathcal C'(H)$, and note that discreteness of $\bs C'$ means that, for every $x\in H$, there are at most finitely many elements $(B,\theta;m)\in\bs C'$ such that $x\in B$. 
If $\bs C'$ consists of only pointed marked balls/horoballs a.s., then it is called a \defstyle{point process of pointed marked balls/horoballs}. 

The mentioned characterization of precompact subsets of $\mathcal C(H)$ and $\mathcal C'(H)$ implies the following lemma. The proof is left to the reader. 

\begin{lemma}
	\label{lem:tight}
	A sequence $(\bs C_n)_n$ of point processes in $\mathcal C(H)$ is tight if and only if, for every $x\in H$, the sequence of random variables $\card{\{(B,\theta)\in\bs C_n: x\in B\}}$, $n=1,2,\ldots$, is tight. By the assumption of compactness of the mark space $\Xi$, the same claim also holds for a sequence of point processes $(\bs C'_n)_n$ in $\mathcal C'(H)$.
\end{lemma}

\section{Perturbed Diamonds Converging to Perturbed Horoballs}
\label{sec:diamond}

In this section, we define perturbed diamonds, which were described heuristically in \Cref{intro:sketch}.
As in \Cref{subsec:notation}, we assume that $G$ and $G'$ are finitely generated, $a>1$ and $a'>1$. We also equip $G''$ with the metric $\rho_c$, where $c=\log a/\log a'$.
\begin{lemma}
	\label{lem:linear}
	There exists an increasing function $f:\mathbb Z^{\geq 0}\to\mathbb Z^{\geq 0}$ and an increasing sequence $(r_j)_{j}$ in $\mathbb Z^{\geq 0}$ such that, by letting $r'_j:=f(r_j)$, one has
	\begin{eqnarray*}
		f(0) &=& 0,\\
		\sup_n v'_{r'_n}/v_{r_n}&<&\infty,\\
		\inf_n v'_{r'_n}/v_{r_n}&>&0,
	\end{eqnarray*}
	and, in addition, $f$ is almost linear with slope $c$ in the sense that
	\begin{equation}
		\label{eq:flinear}
		\forall m: \exists N: \forall n\geq N: \norm{f(n+m)-f(n)-cm}\leq 1.
	\end{equation}
\end{lemma}

Roughly speaking, the \textit{slope} of $f$ is almost $c$, but note that $f(n)-cn$ might diverge. For instance, one might have $f(n) = c n+\log n + O(1)$.

\begin{proof}
	We first construct a function $g:\mathbb Z^{\geq 0}\to \mathbb R$  and $(r_j)_j$ inductively that satisfy the same conditions, and in addition, $v'_{r'_{2n}}\geq v_{r_{2n}}$ and $v'_{r'_{(2n+1)}}\leq v_{r_{(2n+1)}}$ for all $n$.
	Start from $r_0:=0$ and $g(0):=0$. Assume that $r_0,\ldots,r_{2n}$ and $(g(x))_{x\leq r_{2n}}$ are defined. In particular, one has $v'_{g(r_{2n})}\geq v_{r_{2n}}$. For $r_{(2n+1)}$ that will be specified later, and for $r_{2n}<x\leq r_{(2n+1)}$, define $g$ linearly by 
	\begin{equation}
		\label{eq:f}
		g(x):= g(r_{2n})+ (x-r_{2n}) (c-\frac c {n+1}).
	\end{equation}
	Choose $\epsilon>0$ such that $(a'+\epsilon)^{(c-\frac c {n+1})}<a$. For large enough $k$, one has $v'_k\leq (a'+\epsilon)^k$ and $v_k\geq a^k$.
	Therefore, \eqref{eq:f} implies that $v'_{g(x)}/v_x$ converges exponentially to 0. Then, let $r_{(2n+1)}$ be the first time after $r_{2n}$ such that $v'_{g(x)}/v_x$ becomes less than or equal to 1. This guarantees that $v'_{g(r_{(2n+1)})}<v_{r_{(2n+1)}}$. Similarly, for $r_{(2n+2)}$ that will be specified later, and for $r_{(2n+1)}\leq x\leq r_{(2n+2)}$, define $g$ linearly by
	\begin{equation}
		\label{eq:f2}
		g(x):= g(r_{(2n+1)})+ (x-r_{(2n+1)}) (c+\frac c {n+1}).
	\end{equation}
	Choose $\epsilon>0$ such that $(a')^{(c+\frac c {n+1})}>a+\epsilon$.
	For large enough $k$, one has $v'_k\geq (a')^k$ and $v_k\leq (a+\epsilon)^k$. Therefore, \eqref{eq:f2} implies that $v'_{g(x)}/v_x$ converges exponentially to $\infty$. Then, let $r_{(2n+2)}$ be the first time after $r_{(2n+1)}$ such that $v'_{g(x)}/v_x$ becomes larger than or equal to 1. This guarantees that $v'_{g(r_{(2n+2)})}\geq v_{r_{(2n+2)}}$. So, $g$ and $(r_j)_j$ are constructed inductively.
	
	We now verify the conditions for the function $f:=\floor{g}$.
	If the generators of $G$ and $G'$ have size at most $M$, then $v_k\leq v_{k+1}\leq M v_k$ and $v'_k\leq v'_{k+2 c}\leq M^{2c} v'_k$. This implies that, when we considered the first crossing of $v'_{g(x)}/v_x$ from 1 in the above algorithm, the value will be in $[\frac 1 M, M^{2c}]$. 
	
	Also, note that~\eqref{eq:f} and~\eqref{eq:f2} imply $\forall m: \lim_n g(n+m)-g(n)-cm = 0$. This implies the last condition for $f$ and the proof is completed.
\end{proof}

From now on, we fix the function $f$ and sequences $(r_j)_j$ and $(r'_j)_j$ given by \Cref{lem:linear}.

\begin{definition}
	\label{def:diamond}
	A \defstyle{perturbed diamond} with parameter $n$ and center $x'':=(x,x')\in G''$ is the set
	\[
		D_n(x''):= \bigcup_{t=0}^{r_n} \{(y,y'): d(x,y) = r_n-t, d'(x',y')\leq f(t)\}.
	\]
	Also, a \defstyle{(perfect) diamond} is a ball in $G''$ under the metric $\rho_c$.
\end{definition}

\begin{example}
	If $G=G'$, then one can let $r_n=r'_n=n$ and $f(r)=r$. Otherwise, if the assumption of \Cref{prop:growth} holds, then one can let $f(r)=\floor{cr}$. In these cases, the reader might replace perturbed diamonds with ordinary $\rho_c$-balls allover the paper.
\end{example}

The following lemma states the key property of perturbed diamonds needed for \Cref{thm:fixedpriceGeneral}. Roughly speaking, the lemma says that a large perturbed diamond looks like a large perfect diamond, except maybe near the \textit{corners}. So, the limit of large perturbed diamonds is roughly the same as the limit of large diamonds, if the {corners} escape to infinity.

\begin{lemma}[Limits of Perturbed Diamonds]
	\label{lem:perturbedHoroball}
	Assume $x''_n:=(x_n,x'_n)\in G''$ is a sequence such that $d(x_n,o)\to\infty$ and $d'(x'_n,o')\to \infty$. Then, by passing to a subsequence if necessary, $x''_n$ converges to some $\theta''=(\theta,\theta')\in\partial G\times \partial G'$ and	
	$D_n(x''_n)$ converges (in the Fell topology) to either $\emptyset$, $G''$, or a set which is sandwiched between two $\rho_c$-horoballs of type II with center $\theta''$ and slightly different delays; more precisely, between two horoballs of the form $HB(\theta'',\delta-2/c)$ and $HB(\theta'',\delta+1/c)$. In addition, every subsequential limit of $(D_n(x''_n),x''_n)$ is of this form. 
\end{lemma}
\begin{proof}
	By passing to a subsequence, one may assume that $x''_n$ and $D_n(x''_n)$ are convergent.
	Assume $D_n(x''_n)$ does not converge to $\emptyset$ nor to $G''$. So, one might assume that $x_n\to\theta$, $x'_n\to\theta'$ and $D_n(x''_n)\to C$, for some $\theta\in\partial G$, $\theta'\in\partial G'$ and a nontrivial subset $C\subseteq G''$. 
	By the definition of perturbed diamonds, it is straightforward to find a pair of points $q'':=(q,q')\in C$ and $q''_2:=(q,q'_2)\not\in C$ such that $q'_2$ is adjacent to $q'$. 
	Let $\delta:=d_{\theta''}(q'')$, $V^{+}:=HB(\theta'',\delta + 1/c)$ and $V^{-}:=HB(\theta'',\delta - 2/c)$, where $\theta'':=(\theta,\theta')$. We claim that $V^-\subseteq C\subseteq V^+$, which implies the claim of the lemma.  
	
	The fact $q''\in C$ implies that $d'(x'_n,q')\leq f(r_n-d(x_n,q))$ for large enough $n$, which implies that $\alpha_n:= r_n-d(x_n,q)\to \infty$.
	Also, the assumption $q''_2\not\in C$ implies that $d(x'_n,q')+1\geq d(x'_n,q'_2)> f(\alpha_n)$. So, $d(x'_n,q')=f(\alpha_n)$. 
	
	We now prove that $V^-\subseteq C$. Let $y'':=(y,y')\in V^-$ and $\beta_n:=r_n-d(x_n,y)$. The fact $y''\in V^-$ gives that $d_{\theta''}(y'')\leq d_{\theta''}(q'')-2/c$. Hence, for large enough $n$, one has $d_{x''_n}(y'')\leq d_{x''_n}(q'')-1/c$. So,
	\begin{eqnarray*}
		d(x_n,y) + d'(x'_n,y')/c &\leq& d(x_n,q) + d'(x'_n,q')/c -1/c.\\
		\Rightarrow d'(x'_n,y') &\leq& c(\beta_n-\alpha_n) + f(\alpha_n)-1.
	\end{eqnarray*}
	Note that $\alpha_n\to\infty$ and $\beta_n-\alpha_n$ is bounded. Therefore, \eqref{eq:flinear} implies that, for large enough $n$, $d'(x'_n,y') \leq f(\beta_n)$. Thus, $y''\in D_n(x''_n)$. Since this holds for large enough $n$, one obtains that $y''\in C$. So, it is proved that $V^-\subseteq C$.

	We now prove that $C\subseteq V^+$. Let $z''=(z,z')\in C$. So, for large enough $n$, one has $z''\in D_n(x''_n)$; i.e., $d'(x'_n,z')\leq f(\gamma_n)$, where $\gamma_n:=r_n-d(x_n,z)$. So,
	\begin{eqnarray*}
		d_{x''_n}(z'') &=& d(x_n,z) + d'(x'_n,z')/c\\
		&\leq& d(x_n,z) + f(\gamma_n)/c\\
		&=& d(x_n,q)+\alpha_n-\gamma_n+f(\gamma_n)/c\\
		&\leq& d(x_n,q)+f(\alpha_n)/c + 1/c\\
		&=& d(x_n,q)+d(x'_n,q')/c + 1/c\\
		&=& d_{x''_n}(q'') + 1/c,
	\end{eqnarray*}
	where the last inequality holds for large enough $n$ by~\eqref{eq:flinear} (noting that $\gamma_n\to\infty$ and $\alpha_n-\gamma_n$ is bounded). By letting $n\to\infty$, one obtains that $d_{\theta''}(z'')\leq d_{\theta''}(q'')+1/c$; i.e., $z''\in V^+$. So, the claim is proved.
\end{proof}

\begin{definition}
	\label{def:perturbedHoroball}
	A \defstyle{perturbed pointed horoball of type II} is a pointed set $(B,\theta'')$ that is a limit of pointed perturbed diamonds $(D_n(x''_n),x''_n)$ that satisfy the assumptions of \Cref{lem:perturbedHoroball}. In particular, $\theta''\in\partial G\times \partial G'$ and $B$ is sandwiched between two horoballs of the form $HB(\theta'',\delta-2/c)$ and $HB(\theta'',\delta+1/c)$. The whole pointed space $(G'',\theta'')$, pointed at an arbitrary $\theta''\in\partial G\times \partial G'$, is also considered as a perturbed horoball (with infinite delay).
\end{definition}

We obtain the following corollaries of the above lemma for convergence of point processes of perturbed diamonds. To state the lemmas, given $T<\infty$, let $A_{n,T}$ be the set of pointed perturbed diamonds with parameter $n$ whose center $(x,x')$ satisfies either of the following:
\begin{eqnarray*}
	&&d(o,x)< r_n+T\ \text{ and } \ d'(o',x')< T,\\
	\text{or }&& d'(o',x')< r'_n+T\ \text{ and }\ d(o,x)<T.
\end{eqnarray*}
Roughly speaking, these conditions mean that the diamond is close to a (perturbed) horoball of type I.


\begin{lemma}
	\label{lem:perturbedHoroball2}
	For every $n$, let $C_n\subseteq\mathcal C(G'')$ be a discrete set of pointed perturbed diamonds with parameter $n$. Consider the set $A_{n,T}$ 
	defined before the lemma. Assume that:
	\begin{eqnarray*}
		\forall T: \lim_n \card{C_n \cap A_{n,T}} &=& 0.
	\end{eqnarray*}
%
%
%
	Then every subsequential limit of $(C_n)_n$ in $\mathcal C(G'')$ constitutes only of perturbed horoballs of type II (possibly with infinite delay).
\end{lemma}
\begin{proof}
	Assume $C_n\to C$ and let $(B,\theta'')$ be an element of $C$. The definition of $\mathcal C(G'')$ implies that there exists a sequence $x''_n=(x_n,x'_n)$ in $G''$ such that $(D_n(x''_n),x''_n)\in C_n$, $D_n(x''_n)\to B$ and $x''_n\to\theta''$. 	
	Given $T<\infty$, the assumption on $A_{n,T}$ implies that, for large enough $n$, 
	\begin{eqnarray*}
		&&d(o,x_n)\geq r_n+T\ \text{ or } \ d'(o',x'_n)\geq T,\\
		\text{and }&& d'(o',x'_n)\geq r'_n+T\ \text{ or }\ d(o,x_n)\geq T.
	\end{eqnarray*}
	If one of the left inequalities happens, then $D_n(x'')$ is far from $o''$. This is impossible for large enough $T$ (since $B\neq\emptyset$ by the definition of $\mathcal C(G'')$). Thus, 
	$d(o,x_n)\geq T$ and $d'(o',x'_n)\geq T$. This proves that $d(o,x_n)\to\infty$ and $d'(o',x'_n)\to\infty$. So, \Cref{lem:perturbedHoroball} implies that $V$ is a perturbed horoball of type II, possibly with infinite delay. So, the claim is proved. 
\end{proof}

\begin{lemma}
	\label{lem:perturbedHoroball3}
	For every $n$, let $\bs C_n$ be a point process in $\mathcal C(G'')$ which constitutes only of perturbed diamonds with parameter $n$. Consider the set $A_{n,T}$ 
	defined before \Cref{lem:perturbedHoroball2}. Assume that:
	\begin{eqnarray*}
		\forall T: \lim_n \myprob{C_n\cap A_{n,T}\neq\emptyset} &=& 0.
	\end{eqnarray*}
	%
	%
	%
	Then, every subsequential limit of $(\bs C_n)_n$ in $\mathcal C(G'')$ constitutes only of perturbed horoballs of type II (possibly with infinite delay) a.s.
\end{lemma}
\begin{proof}
	The claim is implied by \Cref{lem:perturbedHoroball2} and Skorokhod's representation theorem. More precisely, by the latter, we may choose a coupling of $\bs C_1,\bs C_2,\ldots$ such that $\bs C_n\to \bs C$ a.s. The assumption implies that, given any $T$, the probability of the event that only finitely many of the events $C_n\cap A_{n,T}=\emptyset$ occur, is zero. So, almost surely, the first condition of \Cref{lem:perturbedHoroball2} is satisfied by possibly passing to a subsequence (the subsequence may depend on the realization of $\bs C, \bs C_1,\bs C_2,\ldots$). Therefore, \Cref{lem:perturbedHoroball} implies that $\bs C$ contains only perturbed horoballs of type II (possibly with infinite delay), a.s.
\end{proof}

\section{Induction Formula for Marked Point Processes (with Multiple Points)}
\label{sec:induction}

Let $\Gamma$ be a countable group equipped with a boundedly-finite left-invariant metric $d$. Let $o$ be the neutral element of $\Gamma$. A \textit{multi-set} in $\Gamma$ is an unordered collection of points of $\Gamma$ where each point can appear more than once. In this section, we need to distinguish between multiple points, for reasons similar to the requirement of being essentially free for the action in the induction formula~\eqref{eq:induction}. To achieve this, we use the notion of \textit{marked point processes} from stochastic geometry (see e.g., Chapter~3 of~\cite{bookScWe08}), recalled below.

\begin{definition}[Marked Subsets]
	\label{def:multiset}
	Fix a compact metric space $\Xi$, which is considered as the \textit{mark space} here.
	A \defstyle{discrete marked subset} of $\Gamma$, or for short, a \defstyle{marked subset} of $\Gamma$, is a discrete subset of $\Gamma\times\Xi$. For $x\in \Gamma$, the \defstyle{multiplicity} of $x$ in $S$ is
	\[
		m_S(x):=\card{S\cap(\{x\}\times \Xi)},	
	\]
	which is finite by compactness of $\Xi$. Let $\m(\Gamma)$ denote the set of discrete marked subsets of $\Gamma$.
\end{definition}

By regarding $\m(\Gamma)$ as a subset of the set of integer-valued measures on $\Gamma\times \Xi$, one obtains that $\m(\Gamma)$ is a Borel subset of some Polish space. Note that $\Gamma$ acts on $\m(\Gamma)$ by $h\cdot S:= \{(hx,\xi): (x,\xi)\in S\}$.

\begin{definition}[Stationary Marked Point Process]
	\label{def:markedPP}
	If $\bs S$ is a random element of $\m(\Gamma)$ whose distribution is invariant under the action of $\Gamma$, then $\bs S$ is called a \defstyle{(simple) stationary marked point process} on $\Gamma$. The \defstyle{intensity} of $\bs S$ is $\lambda(\bs S):=\omid{m_{\bs S}(o)}=\omid{\card{\bs S\cap(\{o\}\times\Xi)}}$.
\end{definition}

It should be noted that the term \textit{simple} in this definition means that every {marked point} $(x,\xi)$ appears with multiplicity 1 in $\bs S$. But points $x\in \Gamma$ may appear with more than one mark $\xi\in \Xi$; i.e., one might have $m_{\bs S}(x)>1$.

Given a pmp action $\alpha:\Gamma\curvearrowright (X,\mu)$ of $\Gamma$, one can define a \defstyle{factor marked point process} of $\alpha$ as in \Cref{subsec:cost}; i.e., a measurable function which assigns a discrete marked subset of $\Gamma$ to every element of $X$, in a way that commutes with the actions of $\Gamma$. Since $\mu$ is invariant under the action of $\alpha$, the resulting marked point process is automatically stationary.

Next, we define graphings of marked point processes. We stress that the multiple points are considered as distinct vertices which can be connected to each other, or to other vertices, by edges. This is modeled as follows.

\begin{definition}[Marked Graph]
	\label{def:multiset-graph}
	A \defstyle{(vertex-) marked graph} on $\Gamma$ is a tuple $(S,E)$, where $S\subseteq \Gamma\times \Xi$ is a discrete marked subset and $E$ is a set of undirected edges on $S$ (recall that a point of $\Gamma$ may represent more than one vertex). Let $\mg(\Gamma)$ denote the set of marked graphs on $\Gamma$.
\end{definition}

Every marked graph can be regarded as a discrete subset of $(\Gamma\times \Xi)^2$. So, one can show that $\mg(\Gamma)$ is also a Borel subset of some Polish space. Note that $\Gamma$ acts on $\mg(\Gamma)$ similarly to that of $\m(\Gamma)$.

\begin{definition}[Stationary Random Marked Graph]
	\label{def:multiset-graph-stationary}
	A random element $\bs g:=(\bs S, \bs E)$ of $\mathrm{MG}(\Gamma)$, whose distribution is invariant under the action of $\Gamma$, is called a \defstyle{stationary random marked graph}. Assuming that $0<\lambda(\bs S)<\infty$ (see \Cref{def:markedPP}), the \defstyle{cost} of $\bs g$ is defined by
	\[
	\cost(\bs g):= \frac{1}{2\lambda(\bs S)} \omid{\sum_{o'\in \bs S\cap (\{o\}\times \Xi)}\deg(o',\bs E)}.
	\]
\end{definition}
	\label{rem:multiset-Palm}
	The cost of $\bs g$ is the half expected degree of a \textit{typical vertex} of $\bs S$. This can be formalized by the \textit{Palm distribution}, which is an important notion of stochastic geometry, defined as follows: 
	
	\begin{definition}
		\label{def:palm}
		Assuming $0<\lambda(\bs S)<\infty$,
		the \defstyle{Palm distribution} of $\bs g$ is the probability measure ${\mathbb P_0}$ on $\mg(\Gamma)\times (\Gamma\times \Xi)$ obtained by biasing the distribution of $\bs g$ by $m_{\bs S}(o)$, and then, choosing $\bs o'\in \bs S\cap(\{o\}\times \Xi)$ randomly and uniformly as the root; i.e.,
		\begin{eqnarray*}
			\probPalm{0}{(\bs S, \bs E, \bs o')\in A} := \frac{1}{\lambda(\bs S)}\omid{\sum_{o'\in \bs S\cap (\{o\}\times \Xi)} \identity{A}(\bs S, \bs E, o')},
		\end{eqnarray*}
		for every Borel $A\subseteq \mathrm{MG}(\Gamma)\times (\Gamma\times \Xi)$. One can check that the Palm distribution satisfies the following \defstyle{mass transport principle}: For every Borel function $f:\mg(\Gamma)\times (\Gamma\times \Xi)^2\to\mathbb R^{\geq 0}$, one has
		\begin{eqnarray}
			\label{eq:mtp}
			\omidPalm{0}{\sum_{o''\in \bs S} f(\bs S, \bs E, \bs o', o'')} = \omidPalm{0}{\sum_{o''\in \bs S} f(\bs S, \bs E, o'', \bs o')}.
		\end{eqnarray}
	\end{definition}
	Using the Palm distribution, one can rewrite the cost of $\bs g$ by
	\[
	\cost(\bs g) = \frac 12 \omidPalm{0}{\deg(\bs o')}.
	\]
	\begin{remark}
		The mass transport principle implies that, if $(\bs S, \bs E)$ is connected a.s., then $(\bs S, \bs E, \bs o')$ is a \textit{unimodular} marked graph \cite{processes} under the distribution $\mathbb P_0$.  Similarly, in the disconnected case, the connected component containing $\bs o'$ is unimodular.
	\end{remark}
	

Using the above definitions, one can define the cost of a stationary marked point process $\bs S$ by considering the infimum cost of connected marked graphs with vertex set $\bs S$. But, for the induction formula, we are interested in the more general case where $\bs S$ is a factor of another action $\alpha:\Gamma\curvearrowright(X,\mu)$. 
For instance, in the proof of \Cref{thm:fixedpriceFG}, $\bs S$ is the union of \textit{pointed marked perturbed horoballs}, and we will use the structure of the horoballs and their centers (not just the union of sets) to define a graph.

In the general case, roughly speaking, to each element $x\in X$, one needs to assign a marked graph on $\Gamma$ such that its vertex set is identical to the given marked subset corresponding to $x$. 
More precisely:

\begin{definition}[Induced Cost of Factor Marked Point Processes]
	Let $\alpha:\Gamma\curvearrowright (X,\mu)$ be a pmp action of $\Gamma$ and $\bs S$ be a (simple) marked point process as a factor of $\alpha$. 
	A \defstyle{factor graphing} of $\bs S$ (given that $\bs S$ is also a factor of $\alpha$) is a factor marked graph of $\alpha$ whose vertex set is $\bs S$; i.e., a measurable $\Gamma$-equivariant function from $X$ to $\mathrm{MG}(\Gamma)$ such that the following diagram is commutative:
	\[ 
	\begin{tikzcd}
		& \mathrm{MG}(\Gamma)\arrow{d}\\
		X\arrow{ur}\arrow{r}&\m(\Gamma).
	\end{tikzcd}
	\]
	The \defstyle{induced cost} of $\bs S$, given $\alpha$, is the infimum cost of all factor graphings of $\bs S$ which are connected a.s., and is denoted by $\cost_{\alpha}(\bs S)$.
\end{definition}

In this definition, note that the pushforward of $\mu$ onto $\mathrm{MG}(\Gamma)$ defines a stationary random marked graph, and hence, the cost of the latter is defined by \Cref{def:multiset-graph-stationary}.

We are now ready to prove \Cref{thm:induction}. We provide two proofs. The first one reduces the claim to that of CBERs using the Palm distribution (\Cref{def:palm}). For clarity, a second proof is also provided, which is in fact a translation of the ingredients of the first proof without using CBERs. The reader might safely skip the first proof.

\begin{proof}[First proof of \Cref{thm:induction}]
	%
	Let $\bs S$ be a factor simple marked point process.
	Let $\bar{\bs S}$ be the projection of $\bs S$ on $\Gamma$ (removing multiplicities). By~\eqref{eq:induction}, the claim holds for $\bar{\bs S}$. We deduce the claim for $\bs S$ as follows.
	
	By an abuse of notation, for $x\in X$, we denote by $S(x)$ the marked subset corresponding to $x$ (i.e., we let $S:X\to\m(\Gamma)$ be a the factor map that defines $\bs S$). Let $Y:=Y(S):=\{(x,\xi)\in X\times \Xi: (o,\xi)\in S(x)\}$. In other words, $Y$ describes all choices of $x\in X$ and a point on the fiber $S(x)\cap(\{o\}\times\Xi)$. Let $R$ be the equivalence relation on $Y$ defined as follows: For all $(x,\xi)\in Y$, $h\in \Gamma$ and $\xi'\in\Xi$, if $(h,\xi')\in S(x)$, then we say $(x,\xi)R(h^{-1}x,\xi')$. Note that $(o,\xi')\in S(h^{-1}x)$, which guarantees that $(h^{-1}x,\xi')\in Y$. Note also that if the stabilizer of $x$ is trivial, then $S(x)$ is in bijection with the $R$-equivalence class of $(x,\xi)$ under the map $(h,\xi')\mapsto (h^{-1}x,\xi')$.
	Palm 
	theory defines a probability measure $\mu_0$ on $Y$ as follows: 
	\[
		\mu_0(A):=\frac 1{\lambda(\bs S)}\int_X \sum_{\xi: (o,\xi)\in S(x)} \identity{A}(x,\xi)d\mu(x),
	\]
	for Borel sets $A\subseteq Y$ (compare $\mu_0$ and $\mathbb P_0$ defined in \Cref{rem:multiset-Palm}).
	It can be seen that $R$ is a CBER and $\mu_0$ is an invariant probability measure for $R$.
	
	Now, we construct a bijection between graphings of $(Y,R,\mu_0)$ and factor graphings of $\bs S$.
	Assume $D$ is a graphing of $(Y,R,\mu_0)$; i.e., a symmetric Borel subset of $R$. Given $x\in X$, define a graph with vertex set $S(x)$ as follows: put an edge between $(h_1,\xi'_1)\in S(x)$ and $(h_2,\xi'_2)\in S(x)$ when $((h_1^{-1}x,\xi'_1),(h_2^{-1}x,\xi'_2))\in D$. This is a factor graphing of $\bs S$ (proving the measurability of this construction is omitted for brevity). 
	Also, by the definition of the Palm distributions, one can show that the cost of these two graphings are equal. 
	If the stabilizer of $x$ is trivial, then the last construction is invertible. This shows that graphings of $(Y,R,\mu_0)$ are in bijection with factor graphings of $\bs S$ and $\cost_{\alpha}(\bs S) = \cost(Y,R,\mu_0)$.
	
	Since $\Xi$ is compact, one finds a Borel bijection $\iota: \Xi\to [0,1]$. This allows one to choose \textit{the smallest} point in every fiber as follows: $Y':=\{(x,\xi)\in Y: \forall \xi': (x,\xi')\in Y\Rightarrow \iota(\xi')\geq \iota(\xi)\}$. The induction formula for CBERs (Proposition~II.6 of~\cite{Ga00cost}) provides an equation between $\cost(Y,R,\mu_0)$ and $\cost(Y',R',\mu'_0)$, where the latter is the induced measured CBER on $Y'$. If $\bar S(x)\subseteq \Gamma$ denotes the projection of $S(x)$ on $\Gamma$ (removing multiplicities of the points), then there is a natural bijection between $Y'$ and $Y(\bar S(\cdot))$. In addition, this bijection provides an isomorphism of $(Y',R',\mu'_0)$ and the corresponding CBER on $Y(\bar S)$. By this fact and the mentioned induction formula between $Y'$ and $Y$, one deduces that 
	\[
		\lambda(\bs S)(\cost_{\alpha}(\bs S)-1)=\lambda(\bar{\bs S})(\cost_{\alpha}(\bar{\bs S})-1).
	\]
	The right hand side is equal to $\cost(\alpha)-1$ by~\eqref{eq:induction}. So, the claim is proved.
\end{proof}

\begin{proof}[Second Proof of \Cref{thm:induction}]
	We provide an alternate proof of \Cref{thm:induction} without relying on CBERs. This is in fact the combination of the first proof and a translation of the proof of the induction lemma. 
	
	Let $\pi:\Gamma\times\Xi\to \Gamma$ denote the projection and $\bar{\bs S}:=\pi(\bs S)$ (by removing multiplicities). By~\eqref{eq:induction}, the claim holds for $\bar{\bs S}$. We deduce the claim for $\bs S$ by showing that 
	\begin{equation}
		\label{eq:induction2}
		\lambda(\bs S)(\cost_{\alpha}(\bs S)-1)=\lambda(\bar{\bs S})(\cost_{\alpha}(\bar{\bs S})-1).
	\end{equation}
	
	In what follows, we let $x\in X$ be arbitrary and consider the corresponding sets $S\subseteq \Gamma\times\Xi$ and $\bar S:=\pi(S)$ (given the factor map that constructs $\bs S$ as a function of $x$). In the next constructions, to break the ties, we will use the new marking $u$ defined as follows. Let $\kappa:X\times\Xi \to [0,1]$ be a Borel injective function. 
	For every $y'=(y,\xi)\in S$, let $u(y'):=u_x(y'):=\kappa(y^{-1}x,\xi)$. Note that this marking depends on $x$ as well, and is equivariant under the action of $\Gamma$. The same applies to all of the next constructions. If the stabilizer of $x$ is trivial, then the values of $u$ on all points of $S$ are distinct. This allows one to break the ties in the next constructions. From now on, we assume that $x$ is chosen such that it has trivial stabilizer.
	
	Let $S'\subseteq S$ be the set of $(y,\xi)\in S$ that have the smallest value of $u$ among the points of the fiber $S\cap (\{y\}\times \Xi)$. So, $\pi$ maps $S'$ bijectively onto $\bar S$. We are now ready to prove the $\leq$ and $\geq$ inequalities for the two sides of~\eqref{eq:induction2}.
	
	($\leq$). Assume an edge set $E_1$ is given on $\bar S$ as a function of $x$, such that $x\mapsto (\bar S, E_1)$ is a factor graphing (of $\bar{\bs S}$) and is connected (for $\mu$-a.e. $x$). We construct another set of edges with vertex set $S$ as follows. For every edge $(y_1,y_2)\in E_1$, add its inverse image under $\restrict{\pi}{S'}$. Then, for every $y'=(y,\xi)\in S'$, connect $y'$ to each of the other points of the fiber $S\cap(\{y\}\times \Xi)$ by an edge. Let $E_2$ be the resulting set of edges (proving the measurability of the constructions is omitted for brevity). It is clear that, if $E_1$ is connected, then so is $E_2$. We now compare their costs. 
	By calculating the sum in the next equation, one obtains (here, bold symbols like $\bs E_1, \bs E_2,\ldots$ are used when $x$ is chosen randomly with distribution $\mu$):
	\begin{eqnarray*}
		2\lambda(\bs S)\cost(\bs E_2) &=& \omid{\sum_{o'\in \bs S\cap (\{o\}\times \Xi)}\deg(o',\bs E_2)}\\
		&=& \omid{2\times \identity{\bar{\bs S}}(o)(m_{\bs S}(o)-1) + \deg(o,\bs E_1)}\\
		&=& \omid{2(m_{\bs S}(o)-m_{\bar{\bs S}}(o)) + \deg(o,\bs E_1)}\\
		&=& 2\lambda(\bs S)-2\lambda(\bar{\bs S}) + 2\lambda(\bar{\bs S})\cost(\bs E_1).
	\end{eqnarray*}
	This proves that $\lambda(\bs S)(\cost_{\alpha}(\bs S)-1) \leq \lambda(\bar{\bs S})(\cost_{\alpha}(\bar{\bs S})-1)$.
	
	($\geq$). Conversely, assume that an edge set $E_3$ is given on $S$ as a function of $x$, such that $x\mapsto (S, E_3)$ is a factor graphing (of ${\bs S}$) and is connected (for $\mu$-a.e. $x$).
	We construct another set of edges on $\bar S$ as follows. We cannot consider the naive projection of $E_3$ under $\pi$, since this projection increases the expected degree. So, we need less edges.
	
	For every point $y'=(y,\xi)\in S$, let $\varphi(y')$ be the point of $S'$ which is closest to $y'$ under the graph-distance metric corresponding to $E_3$ (break the ties using the marks $u$). If $y'\in S\setminus S'$, let $\psi(y')$ be a neighbor of $y'$ (in $E_3$) which is on a geodesic between $y'$ and $\varphi(y')$ (break the ties using $u$). By connecting $y'$ to $\psi(y')$ with a directed edge, a forest $F$ on $S$ is obtained. The connected components of $F$ are precisely the inverse images of $\varphi$. Let $d^+(y')\in\{0,1\}$ and $d^-(y')\geq 0$ denote the out-degree and the in-degree of $y'$ in $F$. Let $E_4$ be the graphing on $S'$, defined by putting an edge between $(y_1,\xi_1)$ and $(y_2,\xi_2)$ if and only if they are distinct and there is an edge of $E_3$ between $\varphi^{-1}(y_1,\xi_1)$ and $\varphi^{-1}(y_2,\xi_2)$. Let $E_5$ be the set of edges on $\bar S$ obtained by projecting $E_4$. It is clear that $E_5$ is connected. We now bound its cost. In the next formulas, we use bold symbols when $x$ is chosen randomly with distribution $\mu$. Also, we use $\mathbb P_0, \mathbb E_0$ and $\bs o'$ when the Palm distribution is used (see \Cref{def:palm}).
	Since $S'$ projects bijectively on $\bar S$, one has 
	\begin{eqnarray*}
		\omid{\deg(o,\bs E_5)} &=& \omid{\sum_{o'\in \bs S\cap (\{o\}\times \Xi)} \deg(o',\bs E_4)\identity{\bs S'}(o')}\\
		&=& \lambda(\bs S)\omidPalm{0}{\deg(\bs o',\bs E_4)\identity{\bs S'}(\bs o')}\\
		&\leq& \lambda(\bs S)\omidPalm{0}{\sum_{s\in\varphi^{-1}(\bs o')} \sum_{t\sim s} \identity{\{\varphi(t)\neq \bs o'\}}\identity{\bs S'}(\bs o')},
	\end{eqnarray*}
	where $t\sim s$ means that $t$ is a neighbor of $s$ in $\bs E_3$.
	By the mass transport principle~\eqref{eq:mtp}, we may swap $s$ and $\bs o'$ (which implies that $s=\varphi(\bs o')$), and the last formula is equal to
	\begin{eqnarray*}
		\lambda(\bs S) \omidPalm{0}{\sum_{t\sim \bs o'} \identity{\{\varphi(t)\neq\varphi(\bs o')\}} }
		\leq \lambda(\bs S) \omidPalm{0}{\deg(\bs o',E_3)-d^+(\bs o')-d^-(\bs o')}.
	\end{eqnarray*}
	Again, the mass transport principle implies that 
	\begin{eqnarray*}
		\omidPalm{0}{d^-(\bs o')}&=&\omidPalm{0}{d^+(\bs o')}= \omidPalm{0}{\identity{\{\bs o'\not\in \bs S'\}}}\\
		&=& \frac{1}{\lambda(\bs S)} \omid{\sum_{o'\in \bs S\cap (\{o\}\times \Xi)} \identity{\{o'\not\in \bs S'\}} }\\
		&=& \frac{1}{\lambda(\bs S)} \omid{(m_{\bs S}(o)-1)\identity{o\in\bar{\bs S}}}\\
		&=& \frac{1}{\lambda(\bs S)} \omid{(m_{\bs S}(o)-m_{\bar{\bs S}}(o))}\\
		&=& \frac{1}{\lambda(\bs S)} (\lambda(\bs S)-\lambda(\bar{\bs S})).
	\end{eqnarray*}
	So, the previous inequality implies that
	\[
		\omid{\deg(o,\bs E_5)} \leq \lambda(\bs S) \omidPalm{0}{\deg(\bs o',E_3)}-2\lambda(\bs S)+2\lambda(\bar{\bs S}).
	\]
	This proves that $\lambda(\bs S)(\cost_{\alpha}(\bs S)-1) \geq \lambda(\bar{\bs S})(\cost_{\alpha}(\bar{\bs S})-1)$. So, \eqref{eq:induction2} is proved and the proof is completed.	
\end{proof}

\section{Proof of the Main Theorem}
\label{sec:proof}

We first prove the claim for finitely generated groups:

\begin{theorem}
	\label{thm:fixedpriceFG}
	The product of any two infinite finitely generated groups has fixed price one.
\end{theorem}

\begin{proof}
	As mentioned in the introduction, it is enough to assume that $G$ and $G'$ are both nonamenable. Consider the notations $a, a', c, v_n, v'_n$ and $\rho_c$ of \Cref{subsec:notation}. Since $a>1$ and $a'>1$ by nonamenability, one may consider the function $f:\mathbb Z^{\geq 0}\to\mathbb Z^{\geq 0}$ and the sequences $(r_n)_n$ and $(r'_n)_n$ given by \Cref{lem:linear}.
	
	
	\vspace{2mm}
	\textbf{Defining point processes of perturbed diamonds}.	
	Consider two independent iid marks $\bs u''_1$ and $\bs u''_2$ of $G''$. Also, consider the definition of perturbed diamonds from \Cref{sec:diamond}. Fix $n\in\mathbb N$ and let $v''_n$ be the volume of a perturbed diamond with parameter $n$.
	Let $\Phi_n$ be a Bernoulli point process on $G''$ with parameter $1/v''_n$; e.g., put every $x''\in G''$ in $\Phi_n$ if and only if $\bs u''_1(x'')\leq 1/v''_n$.
	Let $\bs C_n$ be the set of pointed diamonds $\left\{(D_n(x''),x''): x''\in \Phi_n\right\}$. 
	So, $\bs C_n$ is a point process in $\mathcal C(G'')$, where the latter is defined in \Cref{subsec:pp}.
	For every perturbed diamond  $(D_n(x''),x'')\in \bs C_n$, replicate the mark $\bs u''_2(x'')$ to all points of $D_n(x'')$ and let $\bs C'_n$ be the resulting collection of pointed marked perturbed diamonds (with a constant marking on each perturbed diamond). This marking will be used only to distinguish the overlaps of the perturbed diamonds. Note that $\bs C'_n$ is a point process in $\mathcal C'(G'')$, where the latter is defined in \Cref{subsec:pp} (with mark space $\Xi:=[0,1]$).
	
	\vspace{2mm}
	\textbf{The weak limit of the point processes of perturbed balls}.
	We will study the weak limit of $\bs C'_n$ as $n$ tends to infinity along a suitable subsequence. 
	Note that for every $y''\in G''$, the number of marked perturbed diamonds $(B,x'';m)\in \bs C'_n$ such that $y''\in B$ is a binomial random variable with parameters $(v''_n,1/v''_n)$. Therefore, \Cref{lem:tight} implies that the sequence $(\bs C'_n)_n$ is tight as a sequence of point processes in $\mathcal C'(G'')$. So, by refining $(r_j)_j$ and $(r'_j)_j$ if necessary, we may assume that $\bs C'_n$ converges weakly to a point process $\bs C'$ in $\mathcal C'(G'')$. Then, by letting $\bs C$ be the collection of unmarked elements of $\bs C'$, one also has $\bs C_n\to \bs C$ weakly, as point processes on $\mathcal C(G'')$.
	
	\begin{lemma}
		\label{lem:horoball}
		Almost surely, $\bs C\neq\emptyset$ and every element of $\bs C$ is a pointed perturbed horoball of type II (possibly with infinite delay), where the latter is defined in \Cref{def:perturbedHoroball}. 
	\end{lemma}
	For the ease of reading, the proof is given after the proof of the theorem. 
	It is not hard to rule out the pointed perturbed horoballs with infinite delay as well (using nonamenability), but this is not needed in what follows (discreteness of $\bs C$ implies that $\bs C$ can have at most finitely many of such elements).
	One may also modify $\bs C$ to obtain a Poisson process of (non-perturbed) horoballs of type II (by replacing each perturbed horoball with the smallest horoball that contains it), but this is not necessary.
	
	One can show that $\bs C'$ is a Poisson point process on $\mathcal C'(G'')$ and the marks of the different elements of $\bs C'$ are iid (conditionally on the collection of unmarked elements). More precisely:
	\begin{itemize}
		\item $\bs C$ is a Poisson point process on $\mathcal C(G'')$ (with a suitable intensity measure),
		\item (The distribution of) $\bs C'$ is obtained by adding marks to the elements of $\bs C$ as follows: For every pointed perturbed horoball $h:=(B'',\theta'')\in \bs C$, choose a random number $\bs m_0^{(h)}\in[0,1]$ uniformly, and let $\bs m^{(h)}(y''):=\bs m_0^{(h)}$ for every $y''\in B''$. Choose the mentioned random numbers $(\bs m_0^{(h)})_{h\in\bs C}$ independently (given $\bs C$).
	\end{itemize}
	Note that $\bs C$ might be non-simple (if its intensity measure has atoms). In this case, the multiple elements of $\bs C$ appear with different markings in $\bs C'$ a.s.
	
	Let $\bs w''$ be an iid markings of $G''$ with the uniform distribution on $[0,1]$, and independent from $\bs C'$. 
	We will prove that the cost of $(\bs C', \bs w'')$ is one.
	Since the latter is also a weak factor of iid and is essentially free, the claim of the theorem is implied by \Cref{lem:monotonicity} (note that the action on the state space of $(\bs C', \bs w'')$ is continuous).
	
	To prove the claim that $(\bs C', \bs w'')$ has cost one, we will construct a low-cost graphing of every pointed horoball in $\bs C'$, and then, extend it to a low-cost graphing of the union of the horoballs. A difficulty is that we do not know if the horoball containing $o''$ is hyperfinite or not (which is true in {the case of~\cite{FMW}}), since a weak limit of finite unimodular graphs (here, the diamond of $\bs C_n$ containing $o''$) is not necessarily hyperfinite in general. Instead, we will construct the low-cost graphing using the distinguished center of each perturbed horoball, which is already available in $\bs C$.

	Let $\bs S\subseteq G''\times[0,1]$ be the union of the (unpointed) marked perturbed horoballs in $\bs C'$, which is a factor of $\bs C'$. So, the above lemma gives that $\bs S\neq\emptyset$ a.s.
	To show that $(\bs C',\bs w'')$ has cost one, by the induction formula in \Cref{thm:induction}, it is enough to construct a low-cost graphing of $\bs S$ as a factor of $(\bs C',\bs w'')$. 
	%
	
	\vspace{2mm}
	\textbf{Constructing a graphing inside the perturbed horoballs}.
	We recall that for the IPVT cells in products of regular trees, \cite{FMW} constructs a graphing inside the cells by proving the amenability of a typical cell. We do not prove or disprove the amenability of a typical horoball here. Instead, we construct a much simpler graphing directly.
	
	Consider an arbitrary element $V:=(B'',\theta''; m'')\in \bs C'$, where $m''$ is a constant marking of $B''$. 
	By \Cref{lem:horoball}, we may assume that $B''$ is a marked perturbed horoball of type II centered at $\theta''=(\theta,\theta')$, where $\theta\in\partial G$ and $\theta'\in\partial G'$. Consider an arbitrary point $x''=(x,x')\in B''$. Let $\tau^V(x)$ be a neighbor of $x$ in $G$ such that $d_{\theta}(\tau^V(x))= d_{\theta}(x)-1$ (which exists by \Cref{lem:geodesic}). If there are more than one option for $\tau^V(x)$, choose the \textit{smallest} one, in an arbitrary fixed ordering of $G$. Note that no randomness is needed to define $\tau^V(x)$, and $\tau^V(x)$ does not depend on the second coordinate $x'$ of $x''$. One also has $x''_2:=(\tau^V(x),x')\in B''$ since this point is closer to $\theta''$ than $x''$ (this property holds for perturbed diamonds and is thus inherited to perturbed horoballs of type II). By connecting $(x'',m'')\in B''\times [0,1]$ to $(x''_2, m'')$ with a directed edge, a forest is obtained using only horizontal edges (here, we have regarded $m''$ as a number because it is constant). Let $\Pi_1$ be the union of these forests for all $(B'',\theta''; m'')\in\bs C'$, which is a forest on $\bs S$ (proving the measurability of the constructions is omitted for brevity). But $\Pi_1$ is clearly disconnected. We will augment it by adding a small percolation. Before that, let $\gamma^V(x'',m'')$ denote the infinite path obtained by following the out-going edges of $\Pi_1$ starting from $(x'',m'')$. 
	
	
	Consider the following percolation.
	Fix $\epsilon>0$ and any symmetric function $p$ on $G''\times G''$ such that $p$ is equivariant under the action of $G''$, $p$ is positive everywhere and $\sum p(o'',\cdot)=1$.
	Let $\Pi_2$ be a percolation on $G''\times G''$ with parameter $\epsilon p(\cdot,\cdot)$, where the percolation is chosen as a factor of the iid marks $\bs w''$ described above. 
	Let $\pi$ denote the projection from $\bs S\times \bs S$ to $G''\times G''$ obtained by forgetting the marks, and let $\Pi_3:=\Pi_1\cup \pi^{-1}(\Pi_2)$. We claim that $\Pi_3$ is a connected graphing on $\bs S$. 
	Let $V$ be an (unpointed) marked perturbed horoball in $\bs C'$ and consider two points $(x,x'_1,m)$ and $(x,x'_2,m)$ of $V$ with identical first coordinates.
	The paths $\gamma^V(x,x'_1,m)$ and $\gamma^V(x,x'_2,m)$ move parallel to each other, and hence, remain at bounded distance from each other.
	Thus, there exists at least one open path in $\Pi_2$ connecting them a.s. So, $\gamma(x,x'_1,m)$ is connected to $\gamma(x,x'_2,m)$ in $\Pi_3$. This implies that the set of points of the form $(x,\cdot,m)$ in $\bs S'$ (we call this set a \textbf{vertical section} of $V$) belong to the same component of $\Pi_3$. Assume $(x_1,\cdot,m)$ and $(x_2,\cdot,m)$ are two vertical sections of $V$. One can see that there exist infinitely many $x'\in G'$ such that both $(x_1,x',m)$ and $(x_2,x',m)$ belong to $V$ (indeed, it is enough that $d_{\theta'}(x')$ is close enough to $-\infty$, which is possible by \Cref{lem:geodesic} for $\theta'$). Since $p((x_1,x'),(x_2,x'))$ does not depend on $x'$, there are infinitely many open edges of $\Pi_3$ between these two vertical sections. Hence, they are in the same component of $\Pi_3$ a.s. Therefore, any marked horoball $V$ in $\bs C'$ lies entirely in a component of $\Pi_3$ a.s.
	
	\vspace{2mm}
	\textbf{Connecting different perturbed horoballs}. We will show that $\Pi_3$ connects all perturbed horoballs in $\bs C'$ a.s.	
	\begin{lemma}
		\label{lem:infiniteTouching}
		Every two horoballs $B_1$ and $B_2$ of type II have the infinite semi-touching property~\eqref{eq:infiniteTouching}; i.e., there exist sequences $(\xi^{(1)}_j)_j$ and $(\xi^{(2)}_j)_j$ in $B_1$ and $B_2$ such that $\sup_j \rho_c(\xi^{(1)}_j,\xi^{(2)}_j)<\infty$.
	\end{lemma}
	The proof of this lemma is given after the proof of the theorem.
	This lemma implies that any two perturbed horoballs of type II also enjoy the infinite semi-touching property. Note that, if $(\xi^{(1)}_j)_j$ and $(\xi^{(2)}_j)_j$ are two paths within bounded distance, then there is a strictly positive lower bound on $p(\xi^{(1)}_j, \xi^{(2)}_j)$. This implies that infinitely many of the pairs $(\xi^{(1)}_j, \xi^{(2)}_j)$ are open in $\Pi_2$. Hence, any two (unpointed) perturbed marked horoballs of $\bs C'$ are in the same component of $\Pi_3$. So, it is proved that $\Pi_3$ is connected a.s.
	
	\vspace{2mm}
	\textbf{The graphing on the marked horoballs has small cost}.
	We recall the notions of a typical point of $\bs S$ and Palm distribution from \Cref{def:palm} as follows:
	Let $\bs K(o''):= \bs S\cap (\{o''\}\times [0,1])$.
	If $\mathbb P$ denotes the distribution of $(\bs C',\bs w'')$, let ${\mathbb P_0}$ be the probability measure obtained by biasing $\mathbb P$ by $\card{\bs K(o'')}=m_{\bs S}(o'')$, and then choosing a \textit{new marked root} $\bs o''_2\in \bs K(o'')$ randomly and uniformly. 
	Since every point of $\bs S$ has exactly one out-going edge in $\Pi_1$, and by using the mass transport principle, one obtains $\omidPalm{0}{\deg(\bs o''_2,\Pi_1)} = 2$. Since the parameter $\epsilon$ of the percolation is arbitrarily small and the multiplicity of every point is a Poisson random variable with parameter 1, one obtains that $\omidPalm{0}{\deg(\bs o''_2, \pi^{-1}(\Pi_2))}$ is arbitrarily close to 0. Therefore, 	
	$\frac 12 \omidPalm{0}{\mathrm{deg}(\bs o''_2,\Pi_3)}$ is arbitrarily close to 1.
	
	As already mentioned, the proof of the theorem is completed. To recall the arguments in the backward direction, the induction formula in \Cref{thm:induction} implies that $(\bs C',\bs w'')$ has cost one. Since the latter is a weak factor of iid, it has maximum cost. Hence, $G''$ has fixed price one, and the claim is proved.
\end{proof}

Now, we prove \Cref{lem:horoball,lem:infiniteTouching}, which were stated in the above proof. 

\begin{proof}[Proof of \Cref{lem:horoball}]
	Let $s_t:=v_t-v_{t-1}$ denote the volume of the sphere of radius $n$. The assumption of nonamenability implies that, for some $\epsilon>0$, one has $\forall t: s_t> \epsilon v_t$. Therefore, $v_{t}/v_{t-1}\geq 1/(1-\epsilon)$. We may assume similarly that $v'_t/v'_{t-1}\geq 1/(1-\epsilon)$.
	
	First, we prove that $\bs C\neq\emptyset$ a.s. Fix $T<\infty$ and let $E_T\subseteq\mathcal C(G'')$ be the set of pointed sets that intersect $\{o\}\times \oball{T}{o'}$. It is enough to show that $\lim_T \myprob{\bs C\cap E_T=\emptyset} = 0$. Since $E_T$ is clopen, this is equivalent to showing that $\lim_T \lim_n \myprob{\bs C_n\cap E_T=\emptyset} = 0$. Note that $\card{\bs C_n\cap E_T}$ is a binomial random variable with parameters $(\card{E'_{n,T}}, 1/v''_n)$, where $E'_{n,T}$ is the set of perturbed diamonds of parameter $n$ that are in $E_T$. So, it is enough to prove that $\forall n: \card{E'_{n,T}}/v''_n\geq (1-\epsilon)^{-T}$. By considering the distance $t$ of $o$ from the first coordinate of the center of the perturbed diamonds, one gets
	\begin{eqnarray*}
		\card{E'_{n,T}} = \sum_{t=0}^{r_n} s_t v'_{f(r_n-t)+T}\geq \frac 1{(1-\epsilon)^T} \sum_{t=0}^{r_n} s_t v'_{f(r_n-t)} = \frac 1{(1-\epsilon)^T} v''_n.
	\end{eqnarray*}
	So, it is proved that $\bs C\neq\emptyset$ a.s.
	
	For the second claim, we will use \Cref{lem:perturbedHoroball3}. 
	Fix $T\in\mathbb N$ and consider the set $A_n:=A_{n,T}$ defined before \Cref{lem:perturbedHoroball2}. By \Cref{lem:perturbedHoroball3}, it is enough to prove 
	\begin{equation}
		\label{eq:corners}
		\lim_n \myprob{C_n\cap A_{n}\neq\emptyset} = 0.
	\end{equation}
	Note that $\card{\bs C_n\cap A_{n}}$ is a binomial random variable with parameters $(\card{A_{n}}, 1/v''_n)$. So, it is enough to show that $\card{A_n}/v''_n$ converges to zero.
	One has
	\begin{equation}
		\label{eq:corners2}
		\card{A_n}\leq v_{(r_n+T)} v'_T + v_T v'_{(r'_n+T)}\leq M^T\left(v_{r_n}v'_T + v_T v'_{r'_n} \right),
	\end{equation}
	where $M$ is any number that is larger than the sizes of the generators of $G$ and $G'$. So, it is enough to show that
	\begin{equation}
		\label{eq:corners3}
		\lim_n \frac{v''_n}{\max\{v_{r_n}, v'_{r'_n} \}} = \infty.
	\end{equation}
	By \Cref{lem:linear}, $M$ can be chosen large enough such that $\forall k: v'_{r'_k}/v_{r_k} \in (\frac 1 M, M)$. Now,
	\begin{eqnarray*}
		v''_n &=& \sum_{t=0}^{r_n} s_{(r_n-t)} v'_{f(t)} \geq \epsilon \sum_{t=0}^{r_n} v_{(r_n-t)} v'_{f(t)}\geq \epsilon \sum_{k=0}^n v_{(r_n-r_k)} v'_{r'_k}\\
		&\geq & \frac{\epsilon} M \sum_{k=0}^n v_{(r_n-r_k)} v_{r_k}\geq \frac{\epsilon} M \sum_{k=0}^n v_{r_n} = \frac{\epsilon n} M v_{r_n}\geq \frac{\epsilon n}{M^2} v'_{r'_n}.
	\end{eqnarray*}
	This proves~\eqref{eq:corners3} and the proof is completed. To recall: \eqref{eq:corners3} and~\eqref{eq:corners2} imply that $\norm{A_n}/v''_n\to 0$, which implies~\eqref{eq:corners}, and the claim is implied by \Cref{lem:perturbedHoroball3}.
\end{proof}

\begin{proof}[Proof of \Cref{lem:infiniteTouching}]
	For $i=1,2$, let $B''_i$ be a horoball of type II centered at $\theta''_i:=(\theta_i,\theta'_i)$, where $\theta_i\in\partial G$ and $\theta'_i\in\partial G'$. We will construct two paths in $B''_1$ and $B''_2$ that remain in bounded distance from each other. The construction has the same idea as the paths shown in Figure~\ref{fig:type} in the plane, and is formalized below. 
	A similar claim was proved in~\cite{FMW} for the IPVT cells, in products of regular trees, using the stabilizer of the centers and \textit{double-recurrence}. We observe that, when IPVT cells are replace by horoballs, this idea can be simplified by constructing the paths directly without using stabilizers or other complicated techniques (we do not prove double-recurrence here).
	

	Choose two arbitrary points $x''_i:=(x_i,x'_i)\in B''_i$, $i=1,2$. Let $\eta=(\eta_1,\ldots,\eta_k)$ be a path in $G$ starting from $x_2$ and ending in $x_1$ (for suitable $k$). Using \Cref{lem:geodesic}, continue $\eta$ to obtain an infinite path such that $d_{\theta_1}(\eta_j)$ is strictly decreasing for $j\geq k$. 
	Let $\eta'=(\eta'_1,\ldots,\eta'_{k'})$ be a path in $G'$ starting from $x'_1$ and ending in $x'_2$. Continue $\eta'$ to obtain an infinite path such that $d_{\theta'_2}(\eta'_j)$ is strictly decreasing for $j\geq k'$. These paths are depicted in the following figure (note that $\eta$ and $\eta'$ do not necessarily converge to $\theta_1$ and $\theta'_2$):
	\begin{center}
	\begin{tikzpicture}
		\clip (-2.2,2.48) rectangle (7.2,-.5);
		\begin{scope}
			\draw[thick] (0,0) circle(2cm);
			
			\coordinate (P) at (150:2cm); 
			\filldraw[black] (P) circle(2pt) node[above left] {$\theta_1$};
			\coordinate (Q) at (120:2cm); 
			\filldraw[black] (Q) circle(2pt) node[above left] {$\theta_2$};
			\draw[black, thick] (0,0) -- (P);
			\node at (1.7,1.7) {$\partial G$};
			\node at (-.7,.6) {$\eta$};
			
			\coordinate (start) at (.7,.7);
			\coordinate (end)    at (0,0);	
			\coordinate (ctrl) at (-30:2cm);
			\coordinate (S) at ($0.5*(ctrl)$);
			\coordinate (T) at ($(start)+0.5*(P)$);
			\draw[black, thick]
			(start)
			.. controls (T) and (S) .. (end);
			
			\filldraw[black] (start) circle(2pt) node[above right] {$x_2$};
			\filldraw[black] (end) circle(2pt) node[below left] {$x_1$};
		\end{scope}
		\begin{scope}[shift={(5,0)},rotate=-30] 
			\draw[thick] (0,0) circle(2cm);
			
			\coordinate (P) at (150:2cm); 
			\filldraw[black] (P) circle(2pt) node[above left] {$\theta'_1$};
			\coordinate (Q) at (120:2cm); 
			\filldraw[black] (Q) circle(2pt) node[above left] {$\theta'_2$};
			\node at (0.8,2.3) {$\partial G'$};
			\node at (-.3,1) {$\eta'$};
			
			\coordinate (start) at (.7,.7);
			\coordinate (end)    at (0,0);	
			\coordinate (ctrl) at (-30:2cm);
			\coordinate (S) at ($-0.5*(ctrl)$);
			\coordinate (T) at ($(start)-0.5*(P)$);
			\draw[black, thick] (start) -- (Q);
			\draw[black, thick]
			(start)
			.. controls (T) and (S) .. (end);
			
			\filldraw[black] (start) circle(2pt) node[above right] {$x'_2$};
			\filldraw[black] (end) circle(2pt) node[below left] {$x'_1$};
		\end{scope}
	\end{tikzpicture}
	\end{center}
	Let $\xi^{(1)}$ and $\xi^{(2)}$ be the paths in $G''$ starting from $x''_1$ and $x''_2$ respectively, such that in each $\xi^{(i)}$, the two coordinates move along $\eta$ and $\eta'$ respectively, but the second coordinates moves with \textit{speed} $c$; more precisely,
	\begin{eqnarray*}
		\xi^{(1)}_j:= (\eta_{j+k}, \eta'_{\floor{c j}}),\quad
		\xi^{(2)}_j:= (\eta_j, \eta'_{\ceil{cj+k'}}).
	\end{eqnarray*}
	The construction gives
	\begin{eqnarray*}
		d_{\theta_1}(\eta_{j+k}) = d_{\theta_1}(x_1) - j, &&
		d_{\theta'_1}(\eta'_{\floor{cj}}) \leq  d_{\theta'_1(x'_1)} + \floor{cj},\\ 
		d_{\theta_2}(\eta_j)\leq d_{\theta_2}(x_2) + j, &&
		d_{\theta'_2}(\eta'_{\ceil{cj+k'}}) = d_{\theta'_2}(x'_2) -\ceil{cj}.
	\end{eqnarray*}
	Recalling the definition of $d_{\theta''_i}$ from~\eqref{eq:d_theta''}, one obtains that $d_{\theta''_i}(\xi^{(i)}_j) \leq d_{\theta''_i}(x''_i)$ for $i\in\{1,2\}$ and $j\in\mathbb N$. Hence, $\xi^{(i)}$ is entirely in $B''_i$.
	Also, it is clear that the distance $\rho_c(\xi^{(1)}_j, \xi^{(2)}_j)$ is bounded as $j$ grows. So, the claim is proved.
\end{proof}

Finally, we deduce \Cref{thm:fixedpriceGeneral} as follows.

\begin{proof}[Proof of \Cref{thm:fixedpriceGeneral}]
	Let $G''=G\times G'$ be a product of countable groups. Consider any enumeration $G=\{g_1,g_2,\ldots\}$ and $G'=\{g'_1,g'_2,\ldots\}$ of $G$ and $G'$. Let $G_n$ be the subgroup generated by $g_1,\ldots, g_n$. If $G_n$ is finite for every $n$, then $G$ is amenable and the claim is already known. So, assume that $G$ is nonamenable, and hence, $G_n$ is infinite (and in fact, nonamenable) for large enough $n$. Define $G'_n$ similarly and assume that $G'_n$ is infinite for large enough $n$. Since $G_n$ and $G'_n$ are finitely generated, \Cref{thm:fixedpriceFG} implies that $G''_n:= G_n\times G'_n$ has fixed price 1 for large enough $n$. Note that the subgroups $G''_n$ are nested and $\cup_n G''_n = G''$. So, Lemma~VI.25 of~\cite{Ga00cost} implies that $G''$ has fixed price 1 and the claim is proved. 
\end{proof}


\appendix
\section{Weak Factor and Weak Containment}
\label{ap:weakFactor}

In this appendix, we prove \Cref{lem:monotonicity,lem:WFvsWC}). Recall from~\cite{bookKe10} that $\alpha:\Gamma\curvearrowright (X,\mu)$ is \defstyle{weakly contained} in $\alpha':\Gamma\curvearrowright (X',\mu')$ if for every $\epsilon>0$, $N\in\mathbb N$, Borel sets $A_1,\ldots,A_N$ of $X$ and finite set $F\subseteq \Gamma$, there exist Borel sets $A'_1,\ldots, A'_N$ of $X'$ such that
\begin{equation}
	\label{eq:weakContainment}
	\norm{\mu(g A_i\cap A_j) - \mu'(g A'_i\cap A'_j)}<\epsilon, \quad \forall (i,j), \forall g\in F.
\end{equation}
Note that this notion depends only on the Borel structure of $X$ and $X'$, while the notion of weak factor depends crucially on the topology of $X$.
\begin{proof}[Proof of \Cref{lem:monotonicity}]
	We provide two proofs.
	
	\textit{First Proof}. 
	\Cref{lem:WFvsWC} implies that $\alpha$ is weakly contained in $\alpha'$. So, the claim follows from Kechris's result (Corollary 10.14 of~\cite{bookKe10}).
	
	\textit{Second Proof}.
	First, assume that $X$ is compact. When $X=K^\Gamma$, the claim is just the discrete case of Theorem~5.10 of~\cite{AbMe22} (see also Theorem~2.23 of~\cite{FMW}). The general compact case follows from the fact that \textit{symbolic dynamics} provides a continuous factor map from $X$ to a closed subset of $X^\Gamma$ (i.e., $\iota(x)(g):=g^{-1}x$), and hence, the notion of \textit{being a weak factor} is the same on $X$ and $X^\Gamma$.
	%
	%
	If $X$ is not compact, \Cref{lem:extension} below shows that $\alpha$ can be extended to a continuous pmp action $\beta$ on a metrizable compactification $Y$ of $X$, where $\mu$ is extended by the zero measure outside $X$. This implies that $\cost(\alpha)=\cost(\beta)$. Also, $\beta$ is a weak factor of $\alpha'$. So, the claim follows from the compact case.
\end{proof}

\begin{proof}[Proof of \Cref{lem:WFvsWC}]
	\ref{lem:WFvsWC:1}.
	Let $\varphi_n:X'\to X$ be factor maps such that $\mu_n:=(\varphi_n)_* \mu'$ converges weakly to $\mu$. Assume $A_1,\ldots,A_N$ are Borel subsets of $X$, $\{g_1,\ldots,g_k\}$ is a finite subset of $\Gamma$, and $\epsilon>0$. We will construct $A'_1,\ldots,A'_N$ that satisfy~\eqref{eq:weakContainment}.
	
	For each $i$, one can find $C_i\subseteq A_i\subseteq U_i$ such that $C_i$ is compact, $U_i$ is open and $\mu(U_i\setminus C_i)\leq\epsilon/8$. 
	We call a subset $S\subseteq X$ a \textit{\textit{good} continuity set} if all sets $S, g_1 S,\ldots, g_k S$ are $\mu$-continuity sets (i.e., $\mu(\partial g_j S)=0$). Continuity of the action implies that, given $x\in X$, all except countably many of the closed balls $B_r(x)$ are \textit{good} continuity sets. So, one can cover $C_i$ by the interior of \textit{good} continuity balls contained in $U_i$. Since $C_i$ is compact, finitely many of these sets cover $C_i$. The union $D_i$ of the last balls is a \textit{good} continuity set and satisfies $\mu(D_i\setminus C_i)\leq\epsilon/8$. In particular, $\mu(A_i\Delta D_i)\leq \epsilon/4$. 
	Also, all sets $g_l D_i \cap D_j$ are also $\mu$-continuity sets. Hence, on can choose a large enough $n$ such that:
	$$\forall i,j,l: \norm{\mu_n(g_l D_i \cap D_j)-\mu(g_l D_i \cap D_j)} \leq \epsilon/2.$$
	Since the action preserves $\mu$, one has
	\[
	\mu\big( (g_l D_i \cap D_j)\Delta (g_l A_i \cap A_j)\big) \leq \mu(g_l D_i\Delta g_l A_i) + \mu(D_j\Delta A_j)\leq \frac{\epsilon}{4} + \frac\epsilon 4 = \frac \epsilon 2.
	\]
	By combining the two inequalities, one gets $\norm{\mu_n(g_l D_i \cap D_j)-\mu(g_l A_i \cap A_j)} \leq \epsilon$. Hence, the sets $A'_i:=\varphi_n^{-1}(D_i)$ satisfy the claim.
	
	\ref{lem:WFvsWC:2}. Assume $X$ is compact. Consider the factor map $\iota:X\to X^\Gamma$ defined by $\iota(x)(g):=g^{-1}x$. Since $X$ is compact, Lemma~8 of~\cite{AbWe11} gives factor maps $\Phi_n:X'\to X^\Gamma$ such that $(\Phi_n)_* (\mu')$ converges weakly to $\iota_* \mu$. Since the projection $\pi: X^\Gamma\to X$ (at the identity element of $\Gamma$) is continuous, one obtains that $(\pi\circ\Phi_n)_*(\mu')$ converges weakly to $\mu$, and the claim is proved. The claim regarding the noncompact case is implied by extending $\alpha$ to an action on any metrizable compactification $Y$ of $X$, by letting $\alpha$ be trivial and $\mu$ be zero on $Y\setminus X$.
\end{proof}

\begin{lemma}
	\label{lem:extension}
	If $\alpha$ is a continuous pmp action on $(X,\mu)$, 
	then $\alpha$ can be extended to a continuous pmp action on some metrizable compactification of $X$.
\end{lemma}
\begin{proof}
	If $X$ is locally compact, the one-point compactification works. In the general case, we proceed as follows. Recall that $X$ is Polish. We may assume $d(\cdot,\cdot)<1$.
	Let $(x_i)_i$ be a countable dense set in $X$. The map $\iota: X\to [0,1]^{\Gamma\times\mathbb N}$ defined by $\iota(x):=\big(d(g x, x_i)\big)_{g,i}$ is a homeomorphism of $X$ into a $G_{\delta}$ subset of the Hilbert cube (see Theorem~4.14 of~\cite{bookKe12classical} and its proof). Let $Y:=\overline{\iota(X)}$. For $g\in \Gamma$, the map on $\iota(X)$ that corresponds to $x\mapsto gx$ is just a permutation of the coordinates. So, it extends continuously on $Y$ and a continuous action on $Y$ is obtained. Now, $(Y,\iota_*\mu)$ satisfies the claim.
\end{proof}

\section*{Acknowledgements}
We thank Miklos Abert, Damien Gaboriau, Russell Lyons and Sam Mellick for valuable comments and discussions. In particular, we thank Damien Gaboriau for mentioning that the finitely generated case of \Cref{thm:fixedpriceGeneral} implies the general case, and we thank Russell Lyons for proving \Cref{prop:nonamenable}, which was posed as a problem in earlier versions of the paper.

\bibliography{bib} 
\bibliographystyle{alpha.bst}

\end{document}